\documentclass[preprint,review,11pt]{elsarticle}

\let\today\relax
\makeatletter
\def\ps@pprintTitle{%
    \let\@oddhead\@empty
    \let\@evenhead\@empty
    \def\@oddfoot{\footnotesize\itshape
         {Preprint submitted to arXiv} \hfill\today}%
    \let\@evenfoot\@oddfoot
    }
\makeatother

\usepackage{amsmath,amssymb}
 \usepackage{amsthm}
 \usepackage{bbm}
 \usepackage{graphicx}
\usepackage{color}
\definecolor{cornell-red}{RGB}{179,27,27}
\usepackage{subcaption}
\usepackage{mathtools}
\usepackage{dsfont}
\usepackage{color}
\usepackage{natbib}
    \usepackage{multirow}
\usepackage{pgfplots}
\usepackage{caption,subcaption}
\usepackage[left=1in,right=1in,top=1in,bottom=1in]{geometry}
\usepackage{indentfirst}
\usepackage[lined,boxed,commentsnumbered, ruled]{algorithm2e}   
\usepackage{setspace}
\usepackage{ulem}  

\usepackage{xcolor}
\usepackage[colorlinks = true,linkcolor = blue,urlcolor  = blue,citecolor = blue,anchorcolor = blue]{hyperref}

\usepackage{booktabs}

\usepackage[resetlabels]{multibib}
\newcites{ec}{References}

\usepackage{natbib}

\graphicspath{ {Figures/} }

\makeatletter  
\def\th@remark{%
  \thm@headfont{\bfseries}%
  \normalfont 
  \thm@preskip\topsep \divide\thm@preskip\tw@
  \thm@postskip\thm@preskip
}
\makeatother

\newtheorem{thm}{Theorem}

\newtheorem{prop}[thm]{Proposition}
\newtheorem{cor}[thm]{Corollary}

\theoremstyle{definition}

\newtheorem{example}{Example}

\theoremstyle{remark}
\newtheorem{rem}{Remark}

\DeclareMathOperator*{\argmax}{arg\,max}

\newcommand{\tp}{\top}

\newcommand{\epst}{\widetilde{\varepsilon}}

\newcommand{\E}{\mathbb{E}}
\newcommand{\Prob}{\mathbb{P}}
\newcommand{\N}{\mathbb{N}}
\newcommand{\Q}{\mathbb{Q}}
\newcommand{\R}{\mathbb{R}}

\newcommand{\calB}{\mathcal{B}}
\newcommand{\calD}{\mathcal{D}}

\newcommand{\calP}{\mathcal{P}}
\newcommand{\calS}{\mathcal{S}}
\newcommand{\calX}{\mathcal{X}}
\newcommand{\calY}{\mathcal{Y}}

\newcommand{\xb}{\pmb{x}}
\newcommand{\yb}{\pmb{y}}
\newcommand{\cb}{\pmb{c}}
\newcommand{\db}{\pmb{d}}
\newcommand{\hb}{\pmb{h}}
\newcommand{\qb}{\pmb{q}}
\newcommand{\xib}{\pmb{\xi}}

\newcommand{\Cb}{\pmb{C}}

\newcommand{\Tb}{\pmb{T}}
\newcommand{\Wb}{\pmb{W}}

\newcommand{\zb}{\pmb{z}}

\newcommand{\mub}{\pmb{\mu}}
\newcommand{\rhob}{\pmb{\rho}}
\newcommand{\lambdab}{\pmb{\lambda}}
\newcommand{\sigmab}{\pmb{\sigma}}

\newcommand{\Lbar}{\overline{L}}
\newcommand{\Ubar}{\overline{U}}

\newcommand{\calSt}{\widetilde{\mathcal{S}}}

\newcommand{\cf}{c^{\tiny{\textup{f}}}}
\newcommand{\cv}{c^{\tiny{\textup{v}}}}
\newcommand{\bb}{\pmb{b}}
\newcommand{\wb}{\pmb{w}}
\newcommand{\zetab}{\pmb{\zeta}}
\newcommand{\pib}{\pmb{\pi}}
\newcommand{\etab}{\pmb{\eta}}
\newcommand{\vphib}{\pmb{\varphi}}
\newcommand{\dlb}{\underline{d}}
\newcommand{\dub}{\overline{d}}
\newcommand{\Deltalb}{\underline{\Delta}}
\newcommand{\Deltaub}{\overline{\Delta}}

\usepackage{titlesec}

    \titlespacing{\section}{0pt}{0ex}{0ex}
    \titlespacing{\subsection}{0pt}{0ex}{0ex}
    \titlespacing{\subsubsection}{0pt}{0ex}{0ex}

    \newcommand{\blue}{\textcolor{black}}  
    
    \usepackage[shortlabels]{enumitem}

\begin{document}

\setlength{\abovedisplayskip}{4pt}%
\setlength{\belowdisplayskip}{4pt}%
\setlength{\abovedisplayshortskip}{4pt}%
\setlength{\belowdisplayshortskip}{4pt}%

\begin{frontmatter}

\title{An inexact column-and-constraint generation method to solve two-stage robust optimization problems}

\author[mymainaddress1]{Man Yiu Tsang}
\cortext[cor1]{Corresponding author. }
\ead{mat420@lehigh.edu}
\author[mymainaddress1]{Karmel S.~Shehadeh\corref{cor1}}
\ead{kas720@lehigh.edu}
\author[mymainaddress1]{Frank E.~Curtis}
\ead{frank.e.curtis@lehigh.edu}

\address[mymainaddress1]{Department of Industrial and Systems Engineering, Lehigh University, Bethlehem, PA,  USA}

\begin{abstract}  
\noindent We propose a new inexact column-and-constraint generation (i-C\&CG) method to solve two-stage robust optimization problems.  The method allows solutions to the master problems to be inexact, which is desirable when solving large-scale and/or challenging problems.  It is equipped with a backtracking routine that controls the trade-off between bound improvement and inexactness.  Importantly, this routine allows us to derive theoretical finite convergence guarantees for our i-C\&CG method.  Numerical experiments demonstrate computational advantages of our i-C\&CG method over state-of-the-art column-and-constraint generation methods.
\end{abstract}
\end{frontmatter}

\section{Introduction}

Robust optimization (RO) is a methodology for formulating optimization problems in which some parameters are uncertain, but belong to a given \textit{uncertainty set}. In RO, one optimizes a system by hedging against the worst-case scenario of uncertain parameters within the predefined uncertainty set. For example, two-stage RO models are employed when some decisions are made before the uncertainty is revealed (i.e., first-stage problem) and some are made after the uncertainty is realized (i.e., second-stage problem). Two-stage RO models have received substantial attention in various application domains because of their ability to provide solutions that are robust to perturbations within the uncertainty set \citep{An_et_al:2014, Neyshabouri_Berg:2017, Zhang_Liu:2022}. We refer readers to \cite{Gabrel_et_al:2014} for a recent survey.

Various solution methods have been proposed to obtain exact solutions to two-stage RO models under the master-subproblem framework. In this framework, a master problem and a subproblem are solved alternately. The master problem, as a relaxation of the RO model, provides a lower bound to the true optimal value, whereas the subproblem provides an upper bound. The algorithm terminates when the relative gap between the lower and upper bounds is less than a prescribed tolerance. One popular method is based on Benders' decomposition (BD), which constructs lower approximations of the objective function from a dual perspective, i.e., via second-stage dual variables \citep{Jiang_et_al:2012,Thiele_et_al:2009}. Another solution method is the column-and-constraint generation (C\&CG) method proposed in~\cite{Zeng_Zhao:2013}. Different from BD, C\&CG constructs lower approximations of the objective function from a primal perspective, i.e., via second-stage variables and constraints. Results from \cite{Zeng_Zhao:2013} demonstrate the computational efficiency of C\&CG over BD. Hence, C\&CG has been employed widely to solve RO problems in many application domains \citep{An_et_al:2014, Du_et_al:2020, Ruiz_Conejo:2015, Zugno_Conejo:2015}.

Recent research has been devoted to addressing the computational challenges arising from solving the subproblems within C\&CG and exploring relaxations of the assumptions adopted in \cite{Zeng_Zhao:2013}. Relaxing the relatively complete recourse assumption, Ayoub and Poss \cite{Ayoub_Poss:2016} derived an alternative mixed integer program (MIP) reformulation of the subproblem under a 0-1 polytope uncertainty set. In a similar line of research, Bertsimas and Shtern \cite{Bertsimas_Shtern:2018} developed a feasibility oracle and extended the convergence results in \cite{Zeng_Zhao:2013} from polyhedral to general compact uncertainty sets.

In various RO problems, the master problem is a large-scale (mixed) integer program.  For this or other potential reasons, solving the master problem to optimality in each iteration can be challenging.  However, discussions on computational issues associated with the master problem in the literature are sparse. The recent article  \cite{Tonissen_et_al:2021} proposed an adaptive relative tolerance scheme when solving the master problem. Although computational results from \cite{Tonissen_et_al:2021} suggest an improvement in solution time by adopting their proposed scheme, the work did not provide theoretical guarantees concerning the accuracy and convergence of the proposed modification of C\&CG.

In this paper, we propose a new inexact C\&CG method for solving general two-stage RO problems. In our i-C\&CG method, the master problems only need to be solved to a prescribed relative optimality gap or time limit, which may be the only tractable option when solving large-scale and/or challenging problems in practice.  Our method involves a backtracking routine that controls the trade-off between bound improvement and inexactness. We derive theoretical guarantees and prove finite convergence of our approach, \blue{demonstrating that our i-C\&CG method converges to the exact optimal solution under some parameter settings}. Numerical experiments on a scheduling problem and a \blue{facility location} problem demonstrate the computational advantages of our i-C\&CG method over a state-of-the-art C\&CG method.

The remainder of the paper is organized as follows. In Section \ref{sec:problem}, we present the general two-stage RO problem.  In Section \ref{sec:exact_C&CG}, we discuss the C\&CG method, then, in Section \ref{sec:inexact_C&CG}, we introduce our proposed i-C\&CG method and present its theoretical properties. Finally, we conduct computational experiments on an operating room scheduling problem in Section \ref{sec:num_expt}.

\section{Two-stage RO Problem and Assumptions} \label{sec:problem}

We consider a two-stage stochastic problem using $\xb\in\R^n$ to denote first-stage variables, $\calX$ to denote the first-stage feasible region, and $\yb\in\R^m$ to denote second-stage variables. These variables can be either continuous or discrete.  Let $\xib$ be a random vector defined on the measurable space $(\R^l, \calB)$, where $\calB$ is the Borel $\sigma$-field. Our problem of interest is the two-stage robust linear optimization problem
\allowdisplaybreaks

\begin{equation} \label{prob:2S_robust}
    \upsilon^\star = \min_{\xb\in\calX}\,\,\bigg\{ \cb^\tp \xb + \max_{\xib\in\Xi} \, \min_{\yb\in\calY(\xb,\xib)} \qb^\tp \yb \bigg\},
\end{equation}
where $\Xi\subseteq\R^l$ is an uncertainty set and $\calY(\xb,\xib)=\{\yb\in\R^m_+ \mid \Tb\xb +  \Wb\yb + \Cb\xib \geq \hb\}$ is the second-stage feasibility set. The parameters $\cb\in\R^n$, $\qb\in\R^m$, $\hb\in\R^r$, $\Tb\in\R^{r\times n}$, $\Wb\in\R^{r\times m}$, and $\Cb\in\R^{r\times l}$ are assumed to be known. For notational convenience, we use $Q(\xb,\xib)$ to denote the value of the second-stage (recourse) problem for a given $(\xb,\xib)\in\calX\times\Xi$, i.e., $Q(\xb,\xib)=\min_{\yb\in\calY(\xb,\xib)}\qb^\tp\yb$.

We make a few standard assumptions on problem \eqref{prob:2S_robust}. First, as in \cite{Zeng_Zhao:2013}, we assume that the uncertainty set $\Xi$ is either a finite set or a bounded polyhedron (i.e., a polytope). This is a mild assumption that is satisfied by various popular uncertainty sets (e.g., the budgeted uncertainty set~\cite{Bertsimas_Sim:2004}). Second, we assume that problem \eqref{prob:2S_robust} has relatively complete recourse, i.e., for any $\xb\in\calX$ and $\xib\in\Xi$, the set $\calY(\xb,\xib)$ is non-empty. Third, we assume problem \eqref{prob:2S_robust} has an optimal solution with finite objective value, which holds trivially when $\calX$ is compact. Various applications, such as facility location problems and scheduling problems, fulfill the second and third assumptions \citep{An_et_al:2014,Neyshabouri_Berg:2017}. \blue{Finally, we assume that for all $\xb\in\calX$ and $\xib\in\Xi$, the objective function value has $\cb^\tp\xb+Q(\xb,\xib) \geq K$ for some $K>0$. This is satisfied in many settings in which $\cb^\tp\xb+Q(\xb,\xib)$ represents actual implementation costs.}

\begin{rem}
Although our discussions are based on the two-stage RO problem, we remark that several two-stage distributionally robust optimization problems can be reformulated in the form of~\eqref{prob:2S_robust}; see \ref{appdx:DRO}. Thus, our proposed i-C\&CG method can solve such problems as well. 
\end{rem}

\section{The Column-and-constraint Generation (C\&CG) Method} \label{sec:exact_C&CG}

Algorithm \ref{algo:exact_C&CG} summarizes the steps of the C\&CG method from \cite{Zeng_Zhao:2013}. At each iteration $j$, a master problem \eqref{prob:Exact_C&CG_Master} is solved using a subset of scenarios $\calS\subseteq\Xi$ to obtain a solution $\xb^j$. Since only a subset of the scenarios is considered, the master problem \eqref{prob:Exact_C&CG_Master} is a relaxation of~\eqref{prob:2S_robust}.  Thus, its optimal value serves as a valid lower bound $LB$ on the optimal value $\upsilon^\star$ of~\eqref{prob:2S_robust}.  Second, given the master problem solution $\xb^j$, a scenario $\xib^\star\in\Xi$ is identified and an upper bound $UB$ is computed from solving the subproblem \eqref{prob:Exact_C&CG_Subproblem}. With the updated $LB$ and $UB$, the current relative optimality gap is computed, i.e., $(UB-LB)/UB$. If this gap is less than the prescribed tolerance $\varepsilon$, then the algorithm terminates and returns the (nearly) optimal solution.  Otherwise, the scenario set $\calS$ is enlarged by $\xib^\star$, the master problem is re-solved, and the method continues.  Adding a new scenario in the master problem is equivalent to adding new second-stage variables and constraints.

A key computational challenge for the C\&CG method is the need to solve the (enlarged) master and subproblems to optimality in each iteration. As mentioned earlier, most existing research focuses on solution methods and reformulations for the subproblem.  In this paper, we assume that the maximin subproblem \eqref{prob:Exact_C&CG_Subproblem} can be solved by an optimality oracle (see, e.g., \cite{Bertsimas_Shtern:2018,Zeng_Zhao:2013}). Moreover, if the uncertainty set $\Xi$ is a polytope, then the oracle outputs an extreme point of $\Xi$ as the optimal solution $\xib^\star$ in step~2. Under the assumption that $\Xi$ is either a finite set or a polytope, the C\&CG method terminates in a finite number of iterations \citep{Bertsimas_Shtern:2018,Zeng_Zhao:2013}.

Note that the first-stage problem in many real-life applications involves a large number of integer-valued variables.  For this and other potential reasons, solving the master problem \eqref{prob:Exact_C&CG_Master} to optimality can be challenging. Moreover, once the method encounters a challenging master problem at some iteration, the subsequent master problems are potentially even more challenging to solve because of the enlarged scenario set. On the other hand, if the master problem is only solved to a prescribed (large) relative gap, then the value $\cb^\tp\xb^j+\delta^j$ might not be a valid lower bound on $\upsilon^\star$, so allowing inexactness is not something that can be done naively.  These observations motivate our proposed i-C\&CG method presented in the next section.
\IncMargin{0em}
\begin{algorithm}[t] 
\setstretch{1.1}
\SetKwInOut{Initialization}{Initialization}
\Initialization{Set $LB=0$, $UB=\infty$, $\varepsilon\in[0,1]$, $\calS=\emptyset$, $j=1$.} 
\textbf{1. Master problem.} Solve the master problem:
\begin{subequations}
\begin{align}
\underset{\xb\in\calX,\,\delta}{\text{minimize}} \quad
&  \cb^\tp\xb + \delta \\
\text{subject to} \quad
&  \delta \geq  Q(\xb,\xib),\quad \forall \xib\in\calS, \\
&  \cb^\tp\xb + \delta \geq 0.
\end{align}  \label{prob:Exact_C&CG_Master}%
\end{subequations}
\hspace{1.6mm}Record the optimal solution $(\xb^j,\delta^j)$ and set $LB \leftarrow  \cb^\tp\xb^j + \delta^j$. \\
\textbf{2. Subproblem.} Solve the subproblem for fixed $\xb=\xb^j$:
\begin{equation}   \label{prob:Exact_C&CG_Subproblem}
    \blue{D^j} = \max_{\xib\in\Xi} \min_{\yb\in\calY(\xb^j,\xib)} \qb^\tp \yb .
\end{equation}
\hspace{1.6mm}Record the optimal solution $\xib^\star$ and value $D^j$.\\
\hspace{2.5mm} Set $UB \leftarrow  \min\big\{UB,\,\cb^\tp\xb^j +D^j\big\}$.\\
\textbf{3. Optimality test.} If $(UB-LB)/UB<\varepsilon$, then terminate and return $\xb^j$. \\
\textbf{4. Scenario set enlargement.}\\
\quad Enlarge the scenario set $\calS \leftarrow \calS\cup\{\xib^\star\}$. \\ \quad Update $j \leftarrow j+1$ and go back to step 1.
\BlankLine
\caption{Column-and-constraint generation (C\&CG) method} \label{algo:exact_C&CG}
\end{algorithm}\DecMargin{1em}

\section{Inexact Column-and-constraint Generation (i-C\&CG) Method} \label{sec:inexact_C&CG}

In this section, we present our i-C\&CG method that aims to address computational difficulties encountered when solving challenging master problems in the C\&CG method. In Section \ref{subsec:inexact_C&CG_algo}, we discuss the details of our i-C\&CG approach. In Section \ref{subsec:inexact_C&CG_theory}, we derive theoretical properties of our i-C\&CG approach and prove a finite convergence guarantee for it. Finally, in Section \ref{subsec:inexact_C&CG_variant}, we demonstrate the computational flexibility of our i-C\&CG method and present some variants of it.

\subsection{i-C\&CG method}  \label{subsec:inexact_C&CG_algo}

Algorithm \ref{algo:inexact_C&CG} summarizes our proposed i-C\&CG method, which shares a similar structure with the previously stated C\&CG method except for a few additional simple updates in step 1 and the additional backtracking routine in step 3.  In step 1.1, the algorithm solves the master problem \eqref{prob:Inexact_C&CG_Master} to within a relative optimality gap tolerance of $\varepsilon_{MP}^j$ and records the best feasible solution. Observe that as long as \eqref{prob:2S_robust} has an optimal solution $\xb^\star$ with finite objective value $\upsilon^\star$ and~$\Lbar$ is finite, the master problem is guaranteed to be feasible. Indeed, one can easily verify that $(\xb,\delta)=(\xb^\star,\max\{\Lbar,\upsilon^\star\}-\cb^\tp\xb^\star)$ is a feasible solution. See Remark~\ref{rem:lower_bound} below for further discussion on the role played by $\Lbar$ in constraint \eqref{eqn:Inexact_C&CG_Master_con2}.  In step 1.2, we obtain a lower bound~$L^j$ and upper bound~$U^j$ on the optimal value $\upsilon^\star_j$ from the solver, which are used in the backtracking routine to ensure convergence.  If $L^j$ is a valid lower bound, then $\ell$ is set as the current iteration index to indicate that it is the most recent iteration with such a valid bound (see Proposition \ref{prop:valid_LB_iC&CG}). In step 1.3, the algorithm sets $\Lbar$ to $U^j$, which may accelerate the lower bound improvement for the next master problem; again, see Remark~\ref{rem:lower_bound} below.  Step 2 of the i-C\&CG method is the same as that for the C\&CG method, where the subproblem is solved to obtain a scenario $\xib^\star$ and compute a valid upper bound.

The backtracking routine in step~3 balances the computational gains and inaccuracies from solving the master problems inexactly.  In particular, we shall show that this backtracking routine can adapt the inexactness tolerance in order to guarantee finite convergence of the method.  The routine can be described as follows.  First, if the \textit{actual} relative gap $(\Ubar-L^\ell)/\Ubar$ is less than the prescribed tolerance $\varepsilon$, then the algorithm terminates and returns $\xb^j$ as the (nearly) optimal solution of the overall problem.  Otherwise, the algorithm proceeds to an \textit{exploitation} or \textit{exploration} step based on the value of the \textit{inexact} relative gap $(\Ubar-U^j)/\Ubar$.  In an exploration step, the algorithm proceeds to step~4 and enlarges the current scenario set as in the C\&CG method.  In an exploitation step, the algorithm exploits knowledge of the current best valid lower bound $L^\ell$ and proceeds to step~1 by solving the master problem based on $L^\ell$ with a reduced relative gap tolerance $\varepsilon_{MP}^j$.  This corrects any inaccuracies from solving prior master problems.

\begin{rem}\label{rem:lower_bound}
In step 1.3, one may set $\Lbar$ to $L^j$ (instead of $U^j$) to ensure the lower bound validity and thus, convergence. In this case, since $L^j$ is a valid lower bound on $\upsilon^\star$, the constraint \eqref{eqn:Inexact_C&CG_Master_con2} is always valid and the algorithm will always set $\ell$ (which tracks the most recent iteration that $L^\ell$ provides a valid lower bound) to the current iteration index $j$. On the other hand, setting $\Lbar$ to $U^j$ (as stated in the algorithm) may accelerate the exploration process and help the solver escape from proving the optimality of a given feasible solution. We have found this to be computationally effective when some master problems in early iterations are challenging and the lower bounds that they provide may be very loose.  In any case, the fact that the exploitation step sets $\Lbar \gets L^\ell$ ensures convergence as we will show in Section \ref{subsec:inexact_C&CG_theory}.
\end{rem}


\begin{rem}
If $\varepsilon_{MP}^j = 0$ and $\epst<\varepsilon$, then $L^j=U^j$ since the master problem is solved to optimality. Thus, in this case, $\Lbar$ is a valid lower bound and the algorithm sets $\ell \gets j$. As a result, the backtracking routine reduces to checking the termination condition.  Hence, our i-C\&CG method is a generalization of the C\&CG method with an additional valid lower bound constraint \eqref{eqn:Inexact_C&CG_Master_con2}.
\end{rem}

\IncMargin{0em}
\begin{algorithm}[t] 
\setstretch{1.1}
\SetKwInOut{Initialization}{Initialization}
\Initialization{$\Lbar\gets0$, $\Ubar\gets\infty$, \blue{$\varepsilon\in[0,1]$}, $\epst\in(0,\varepsilon/(1+\varepsilon))$,  $\{\varepsilon_{MP}^j\in(0,1)\}_{j\in\N}$, $\alpha\in(0,1)$, $\calS\gets\emptyset$, $j\gets1$, $\ell\gets0$.} 
\textbf{1. Master problem.} \\
\textbf{\small\hspace{5.3mm}1.1.} Solve the master problem to within a relative optimality gap of $\varepsilon_{MP}^j$:
\begin{subequations}
\begin{align}
\upsilon^\star_j = \,\, \underset{\xb\in\calX,\,\delta}{\text{minimize}} \quad
&  \cb^\tp\xb + \delta \\
\text{subject to} \quad
&  \delta \geq  Q(\xb,\xib),\quad \forall \xib\in\calS, \label{eqn:Inexact_C&CG_Master_con1} \\
&  \cb^\tp \xb + \delta \geq \Lbar \label{eqn:Inexact_C&CG_Master_con2}.
\end{align}  \label{prob:Inexact_C&CG_Master}%
\end{subequations}
\hspace{3.3mm}Record the best feasible solution $(\xb^j,\delta^j)$ found.  \\
\textbf{\small\hspace{5.3mm}1.2.} Record a lower bound $L^j\geq \Lbar$ and upper bound $U^j= \cb^\tp\xb^j + \delta^j$ of $\upsilon^\star_j$. \\ \hspace{4.3mm} If $L^j> \Lbar$, then set $\ell \leftarrow  j$.  \\
\textbf{\small\hspace{5.3mm}1.3.} Set $\Lbar \leftarrow  U^j$.\\
\textbf{2. Subproblem.} Solve the subproblem \eqref{prob:Exact_C&CG_Subproblem} for fixed $\xb=\xb^j$. \\
\hspace{5.3mm}Record the optimal solution $\xib^\star$ and value $D^j$. \\ \hspace{4.3mm} Set $\Ubar \leftarrow  \min\big\{\Ubar,\, \cb^\tp\xb^j +D^j\big\}$.\\

\textbf{3. Optimality test and backtracking routine.} \\
\hspace{5.3mm}If $(\Ubar-L^\ell)/\Ubar<\varepsilon$, then terminate and return $\xb^j$; otherwise, do the following.
\begin{itemize}
    \item \textbf{Exploitation}: If $(\Ubar-U^j)/\Ubar < \epst$, then set $j\leftarrow \ell$ and $\Lbar\leftarrow L^\ell$. \\ Set $\varepsilon^j_{MP}\leftarrow \alpha \varepsilon^j_{MP}$ for all $j\geq \ell$ and go back to step 1. \vspace{-3mm}
    \item \textbf{Exploration}: If $(\Ubar-U^j)/\Ubar  \geq \epst$, then go to step 4.  \vspace{-3mm}
\end{itemize}
\textbf{4. Scenario set enlargement.}\\
\hspace{5.3mm}Enlarge the scenario set $\calS \leftarrow \calS\cup\{\xib^\star\}$. \\ \hspace{4.0mm} Update $j \leftarrow j+1$ and go back to step 1.
\BlankLine
\caption{Inexact column-and-constraint (i-C\&CG) method} \label{algo:inexact_C&CG}
\end{algorithm}\DecMargin{1em}

\subsection{Theoretical Properties}  \label{subsec:inexact_C&CG_theory}

In this section, we study the theoretical properties of the i-C\&CG method. In particular, we derive an upper bound on the actual relative gap in the presence of inexactness (Proposition \ref{prop:terminate_MIP_gap_iC&CG} and Corollary \ref{coro:terminate_MIP_gap_iC&CG}) and prove a finite convergence property (Theorem \ref{thm:finite_convg_iC&CG}). First, in Proposition \ref{prop:valid_LB_iC&CG}, we show the validity of $L^\ell$ as a lower bound on $\upsilon^\star$ (see \ref{appdx:proof_valid_LB} for a proof).

\begin{prop}  \label{prop:valid_LB_iC&CG}
At any iteration, the value $L^\ell$ is a valid lower bound on the optimal value $\upsilon^\star$ to problem \eqref{prob:2S_robust}, i.e., $L^\ell \leq \upsilon^\star$.
\end{prop}

Since $\Ubar$ is an upper bound on $\upsilon^\star$, together with Proposition \ref{prop:valid_LB_iC&CG}, these results justify the use of the \textit{actual} relative gap $(\Ubar-L^\ell)/\Ubar$ as the termination condition.  That is, if the algorithm terminates, then the actual relative gap computed based on a valid lower and a valid upper bound on $\upsilon^\star$ is less than the prescribed tolerance $\varepsilon$.  Next, Proposition \ref{prop:terminate_MIP_gap_iC&CG} provides an upper bound on the actual relative gap when the algorithm reaches the exploitation step.

\begin{prop}  \label{prop:terminate_MIP_gap_iC&CG}
If the i-C\&CG method reaches the exploitation step at iteration $j$, i.e., the actual relative gap satisfies $(\Ubar-L^\ell)/\Ubar\geq \varepsilon$, but the inexact relative gap satisfies  $(\Ubar-U^j)/\Ubar<\epst$, then the actual relative gap is at most $(1-\epst)^{-1} \prod_{k=\ell}^j (1-\varepsilon_{MP}^k)^{-1}-1$.
\end{prop}

\begin{proof}
First, we consider the case that $j=\ell$. Since $(U^\ell-L^\ell)/U^\ell\leq\varepsilon_{MP}^\ell$ follows from the inexact solution of the master problem~\eqref{prob:Inexact_C&CG_Master}, it follows that $U^\ell\leq L^\ell/(1-\varepsilon_{MP}^\ell)$ and thus,
\begin{equation} \label{eqn_pf:terminate_MIP_gap_iC&CG_0}
  \frac{U^\ell-L^\ell}{L^\ell} \leq \frac{1}{1-\varepsilon_{MP}^\ell}-1 =: \varepsilon' > 0.
\end{equation}
%
Then, we have
\begin{equation} \label{eqn_pf:terminate_MIP_gap_iC&CG_1}
  \frac{U^\ell}{1-\epst} > \Ubar \geq \upsilon^\star \geq L^\ell \geq \frac{U^\ell}{1+\varepsilon'},  
\end{equation}
where the first inequality follows from $(\Ubar-U^\ell)/\Ubar<\epst$, the second inequality follows since $\Ubar$ is an upper bound for $\nu^\star$, the third inequality is a consequence of Proposition~\ref{prop:valid_LB_iC&CG}, and the last inequality follows from \eqref{eqn_pf:terminate_MIP_gap_iC&CG_0}. Using the chain of inequalities \eqref{eqn_pf:terminate_MIP_gap_iC&CG_1}, we obtain the desired inequality:
\begin{equation}\label{eq:new-number}
  \frac{\Ubar-L^\ell}{\Ubar}< \frac{1}{\Ubar}\Bigg( \frac{U^\ell}{1-\epst}-\frac{U^\ell}{1+\varepsilon'} \Bigg) \leq \frac{1+\varepsilon'}{1-\epst}-1,
\end{equation}
where the first inequality follows from the facts that \eqref{eqn_pf:terminate_MIP_gap_iC&CG_1} shows $\Ubar<U^\ell/(1-\epst)$ and $L^\ell\geq U^\ell/(1+\varepsilon')$, and the second inequality follows from the fact that \eqref{eqn_pf:terminate_MIP_gap_iC&CG_1} shows $U^\ell/\Ubar\leq 1+\varepsilon'$.  Hence, from \eqref{eq:new-number}, the desired conclusion follows in the case that $j=\ell$.

Next, we consider the case that $j>\ell$. Note that, for all $k\in\{\ell,\dots,j-1\}$ in step 1, we have $(U^{k+1}-L^{k+1})/U^{k+1} \leq \varepsilon_{MP}^{k+1}$ from solving the master problem to within a relative gap of $\varepsilon_{MP}^{k+1}$. Moreover, step 1.2 implies $L^{k+1}=U^k$ for all $k\in\{\ell,\dots,j-1\}$. Therefore, 
\begin{equation}  \label{eqn_pf:terminate_MIP_gap_iC&CG_2}
  U^{k+1} \leq \frac{L^{k+1}}{1-\varepsilon_{MP}^{k+1}} = \frac{U^k}{1-\varepsilon_{MP}^{k+1}} 
\end{equation}
for all $k\in\{\ell,\dots,j-1\}$. Applying inequality \eqref{eqn_pf:terminate_MIP_gap_iC&CG_2}, we obtain 
$$U^j \leq U^{j-1}\cdot\frac{1}{1-\varepsilon_{MP}^j}\leq \cdots \leq U^\ell\cdot\prod_{k=\ell+1}^j \frac{1}{1-\varepsilon_{MP}^k} \leq L^\ell\cdot \prod_{k=\ell}^j\frac{1}{1-\varepsilon_{MP}^k}, $$
where the last inequality follows from $(U^\ell-L^\ell)/U^\ell\leq\varepsilon_{MP}^\ell$.  This in turn implies
\begin{equation}  \label{eqn_pf:terminate_MIP_gap_iC&CG_3}
\frac{U^j-L^\ell}{L^\ell}\leq \prod_{k=\ell}^j\frac{1}{1-\varepsilon_{MP}^k}-1 =: \varepsilon'' > 0.
\end{equation}
%
Therefore, the chain of inequalities similar to \eqref{eqn_pf:terminate_MIP_gap_iC&CG_1} is as follows:
\begin{equation} \label{eqn_pf:terminate_MIP_gap_iC&CG_4}
  \frac{U^j}{1-\epst} > \Ubar \geq \upsilon^\star \geq L^\ell \geq  \frac{U^j}{1+\varepsilon''},
\end{equation}
where the last inequality follows from \eqref{eqn_pf:terminate_MIP_gap_iC&CG_3}. Hence, we can derive
\begin{equation}\label{eq:newer-number}
\frac{\Ubar-L^\ell}{\Ubar}< \frac{1}{\Ubar} \Bigg( \frac{U^j}{1-\epst}-\frac{U^j}{1+\varepsilon''} \Bigg) \leq \frac{1+\varepsilon''}{1-\epst}-1,
\end{equation}
where the first inequality follows from the fact that \eqref{eqn_pf:terminate_MIP_gap_iC&CG_4} shows $\Ubar<U^j/(1-\epst)$ and $L^\ell>U^j/(1+\varepsilon'')$, and the second inequality follows from the fact that \eqref{eqn_pf:terminate_MIP_gap_iC&CG_4} shows $U^j/\Ubar\leq 1+\varepsilon''$.  Hence, from \eqref{eq:newer-number}, the desired conclusion also follows in the case that $j > \ell$. 
\end{proof}

Proposition \ref{prop:terminate_MIP_gap_iC&CG} quantifies the effect of inexact solves of master problems on the actual relative gap. In Corollary \ref{coro:terminate_MIP_gap_iC&CG}, we provide a bound on the actual relative gap if the algorithm updates $\Lbar$ to $L^j$ in step 1.3, which is equivalent to the C\&CG method with inexact solves of master problems only \cite{Tonissen_et_al:2021} (see \ref{appdx:proof_coro_terminate_MIP_gap_iC&CG} for a proof).

\begin{cor}  \label{coro:terminate_MIP_gap_iC&CG}
Assume that in step 1.3, $\Lbar$ is updated as $L^j$ instead of $U^j$. At iteration $j$, if the actual relative gap satisfies $(\Ubar-L^\ell)/\Ubar\geq \varepsilon$, but the inexact relative gap satisfies $(\Ubar-U^j)/\Ubar<\epst$ in step 3, then the actual relative gap is at most $(1-\epst)^{-1}(1-\varepsilon_{MP}^j)^{-1}-1$.
\end{cor}

Recall that in the backtracking routine (step 3), if the termination condition is not satisfied, then the algorithm proceeds to either the exploration or the exploitation step.  If the inexact relative gap is large (i.e., greater than \blue{or equal to} $\epst$), then the algorithm explores possible new valid lower bounds on $\upsilon^\star$.  Otherwise, only a relatively small improvement in the lower bound could be achieved by the exploration step, and our i-C\&CG method switches to exploiting the best current lower bound. Proposition \ref{prop:terminate_MIP_gap_iC&CG} shows that the actual relative gap is bounded by $\epst$ (the backtracking routine parameter) and $\varepsilon_{MP}^k$ (the relative gap from solving master problems). Therefore, to close the actual relative gap, the algorithm reduces $\varepsilon_{MP}^k$ in every exploitation step. Finally, we leverage the results in Proposition \ref{prop:terminate_MIP_gap_iC&CG} to show the finite convergence of our proposed i-C\&CG method in Theorem~\ref{thm:finite_convg_iC&CG}.

\begin{thm}  \label{thm:finite_convg_iC&CG}
If $\epst<\varepsilon/(1+\varepsilon)$, then Algorithm \ref{algo:inexact_C&CG} terminates in a finite number of iterations.
\end{thm}

\begin{proof}
We first show that every visit to step 4 (via the exploration step) enlarges the scenario set~$\calS$. Equivalently, we want to prove that if $\xib^\star$ in step 2 belongs to the current scenario set $\calS$ (i.e., $\xib^\star\in\calS$), then we will not proceed to the exploration step. Consider master problem \eqref{prob:Inexact_C&CG_Master} at some iteration $j$ with current scenario set $\calS$. Assume that $\xib^\star\in\argmax_{\xib\in\Xi} Q(\xb^j,\xib)$ belongs to $\calS$. Then, by the definition of $U^j$ as an upper bound of $\upsilon^\star_j$, we have
$$\cb^\tp\xb^j + \max_{\xib\in\Xi} Q(\xb^j,\xib)  =\cb^\tp\xb^j + \max_{\xib\in\calS} Q(\xb^j,\xib) \leq U^j,$$
where the first equation follows from $\xib^\star\in\calS$. This implies that $\Ubar$ updated in step 2 satisfies
$$\Ubar-U^j \leq \cb^\tp\xb^j + \max_{\xib\in\Xi} Q(\xb^j,\xib) - U^j \leq 0. $$
Thus, the algorithm will not proceed to the exploration step.

Next, we show that Algorithm \ref{algo:inexact_C&CG} terminates in a finite number of iterations. Note that in the i-C\&CG method, if the termination criterion is not met, then the algorithm visits either the exploitation or the exploration step. Since we proved that every exploration step enlarges the scenario set, under the assumption that $\Xi$ is a finite set or a polytope (with a finite number of extreme points), the number of exploration steps is finite. Hence, to verify the finite convergence property, it suffices to show that for any fixed scenario set $\calSt\subset\Xi$, the master problem \eqref{prob:Inexact_C&CG_Master} is solved at most finitely many times (with possibly different $\Lbar$ values). Suppose, on the contrary, that master problem \eqref{prob:Inexact_C&CG_Master} with scenario set $\calSt$ is solved infinitely many times. This can happen only when neither the termination nor the exploration step is visited. That is, starting from the first re-visit of the master problem with scenario set $\calSt$ via an exploitation step, Algorithm \ref{algo:inexact_C&CG} proceeds with the exploitation step forever. \blue{This implies that the conditions in Proposition \ref{prop:terminate_MIP_gap_iC&CG} are satisfied with $j=\ell$, i.e., $(\Ubar-L^\ell)/\Ubar\geq \varepsilon$, but $(\Ubar-U^\ell)/\Ubar<\epst$. By} Proposition \ref{prop:terminate_MIP_gap_iC&CG}, the actual relative gap is bounded by $1/[(1-\epst)(1-\varepsilon_{MP}^\ell)]-1$.  \blue{However,} since the value of $\varepsilon^\ell_{MP}$ is reduced in each exploitation step, it follows that $\varepsilon^\ell_{MP}$ converges to zero as the i-C\&CG method continues with the exploitation step. Therefore, the actual relative gap $1/[(1-\epst)(1-\varepsilon_{MP}^\ell)]-1$ converges to $\epst/(1-\epst)$, which is less than $\varepsilon$ by our assumption. This implies that after a sufficiently large number of iterations, the termination condition will be satisfied, which contradicts that master problem \eqref{prob:Inexact_C&CG_Master} with scenario set $\calSt$ is solved infinitely many times.
\end{proof}

\begin{rem}
\blue{By setting the actual relative gap tolerance to $\varepsilon=0$, our i-C\&CG method converges to the exact optimal solution to \eqref{prob:2S_robust}.}
\end{rem}

\begin{rem}
We provide the following guidelines for choosing the i-C\&CG method parameters, namely, $\epst$, $\{\varepsilon_{MP}^j\}$, and $\alpha$.  First, recall that $\epst$ determines whether exploitation or exploration is performed. As suggested by Theorem \ref{thm:finite_convg_iC&CG}, one can set \blue{$\epst<\varepsilon/(1+\varepsilon)$} and choose $\epst$ close to this upper bound to favor exploitation. Second, recall that $\{\varepsilon_{MP}^j\}$ and $\alpha$ control the extent of the inexactness allowed by the method. If $\varepsilon_{MP}^j$ is large (e.g., greater than $\varepsilon$), then a value of $\alpha$ that shrinks $\varepsilon_{MP}^j$ at a relatively fast rate is preferred.  In contrast, if $\varepsilon_{MP}^j$ is comparable with $\varepsilon$, then a value of $\alpha$ close to $1$ that shrinks $\varepsilon_{MP}^j$ more slowly is preferred. Finally, we emphasize that there is no one set of parameters that yields the best performance for all problems, and indeed, such a flexibility allows the i-C\&CG method to adapt to problems of different structures.
\end{rem}

\subsection{Variants of i-C\&CG} \label{subsec:inexact_C&CG_variant}

Algorithm \ref{algo:inexact_C&CG} provides a general framework to handle challenging master problems. That is, our proposed i-C\&CG is flexible, allowing users to customize the algorithm for specific problems to achieve better computational performance. In this section, we discuss two variants of our i-C\&CG method that provide additional flexibility for practical use (see \ref{appdx:iCCG_variant}). \blue{We refer readers to \ref{appdx:ROCPC_iCCG_variants} for numerical examples illustrating the potential benefits of these variants. }

The first variant provides additional controls on the trade-off in step 3 by allowing users to impose an exploitation frequency $f^{\mbox{\tiny exploit}}$. Specifically, one can enforce the algorithm to proceed to the exploitation step when $\ell$ does not change for $f^{\mbox{\tiny exploit}}$ iterations, \blue{i.e., when $j-\ell>f^\text{exploit}$}. This mechanism remedies the situation that the valid lower bound information is not exploited for a long time, i.e., $\ell \ll j$. \blue{In such a case, the scenario set is substantially enlarged due to exploration steps, thus a new valid lower bound could potentially be identified by an exploitation step. Therefore, imposing $f^\text{exploit}$ could potentially improve the relative gap convergence rate.
} 

The second variant allows users to impose a time limit $\tau$ for solving the master problem. That is, one can run a solver for the master problem that terminates either when the relative gap of $\varepsilon_{MP}^j$ is reached or when the solution time exceeds $\tau$. Due to this additional source of inexactness, in the exploitation step, the algorithm can increase the time limit by a factor of $\beta>0$, \blue{i.e., $\tau\leftarrow\tau+\beta$}.  As a result, the algorithm establishes an adaptive time limit for solving the master problems that increases with each exploitation step. This variant could be useful to accelerate the lower bound improvement if some intermediate master problems are challenging. \blue{In such situations, the algorithm may spend a significant amount of time to solve a particular master problem. The time limit variant could circumvent this problem, potentially improving the computational performance.}

\section{Numerical Results}\label{sec:num_expt}

In this section, we use a two-stage distributionally robust operating room (OR) scheduling problem recently studied in \cite{Wang_et_al:2019} to compare the performance of the C\&CG and i-C\&CG methods. \blue{In \ref{appdx:additional_results_ROCPCP}, we provide additional computational results on a robust facility location problem.}

\subsection{A distributionally robust operating room scheduling problem (DRORSP)} \label{subsec:DRO_OR_scehduling}
We start by introducing the DRORSP setting as in \cite{Wang_et_al:2019}. Let $I$ be a set of surgeries to schedule and $R$ be a set of ORs. Each surgery $i\in I$ has a random duration $d_i$ where the support of $\db=(d_1,\dots,d_{|I|})^\tp$ is $\Xi=\{\db\in\R^{|I|}\mid \dlb_i\leq d_i\leq\dub_i,\,i\in I\}$. The fixed cost of opening an OR is $\cf$, and a per-unit overtime cost $\cv$ is incurred if an OR operates beyond the working hour $T$. In the DRORSP, given the sets of surgeries $I$ and ORs $R$, the OR manager aims to make the following decisions simultaneously:  (a) decide which OR(s) to open, and (b) assign each surgery to an open OR. The objective is to minimize the fixed cost of opening ORs plus the worst-case expected cost associated with OR overtime. As in \cite{Wang_et_al:2019}, we define the following ambiguity set
\begin{equation} \label{eqn:MAD_ambiguity_set}
   \calP=\big\{\Q\in\calD(\Xi)\mid \E_\Q[d_i]=\mu_i,\, \E_\Q\big[|d_i-\mu_i|\big] \leq \nu_i,\, i\in I\big\},
\end{equation} 
where $\calD(\Xi)$ is the set of probability measures with support $\Xi$. Ambiguity set \eqref{eqn:MAD_ambiguity_set} consists of all probability measures with support $\Xi$ such that, for all $i \in I$, the mean is $\mu_i$ and mean absolute deviation (MAD) is less than $\nu_i$.

For each $r\in R$, we define a binary variable $x_{r}$ that equals $1$ if OR $r$ is open, and is $0$ otherwise. In addition, we define a binary decision variable \blue{$y_{ir}$} that equals $1$ if surgery $i$ is assigned to OR $r$, and is $0$ otherwise. For any $a\in\R$, we write $(a)^+=\max\{a,0\}$. The DRORSP can be stated as:
\begin{subequations}
\begin{align}
\underset{\xb\in\{0,1\}^{|R|},\,\yb\in\{0,1\}^{|I|\times|R|}}{\text{minimize}}\, \quad
&  \sum_{r\in R} \cf x_r + \cv \sup_{\Prob\in\calP} \E_\Prob\Bigg[\sum_{r\in R} \Bigg(\sum_{i\in I} y_{ir}d_i - T \Bigg)^+ \Bigg] \label{eqn:DRORSP_obj}\\
\text{subject to} \hspace{9.5mm} \quad
& y_{ir} \leq x_r,\quad\forall i\in I,\, r\in R,  \label{eqn:DRORSP_con1}\\
& \sum_{r\in R}y_{ir} = 1,\quad\forall i\in I.  \label{eqn:DRORSP_con2}
\end{align} \label{prob:DRORSP}%
\end{subequations}
Objective \eqref{eqn:DRORSP_obj} is a sum of the fixed cost of opening ORs and the worst-case expected overtime cost. Constraint \eqref{eqn:DRORSP_con1} ensures that surgeries are assigned to open ORs only, and constraints \eqref{eqn:DRORSP_con2} require that every surgery is assigned to exactly one OR. 

In \cite{Wang_et_al:2019}, the authors demonstrated the challenges of solving problem \eqref{prob:DRORSP} exactly and developed linear decision rules to approximate solutions to problem instances. For our experiments, we implemented the C\&CG and i-C\&CG methods to solve various instances of problem \eqref{prob:DRORSP}. In \ref{appdx:DRORSP}, we present the detailed derivations of an equivalent reformulation of problem \eqref{prob:DRORSP} and the associated master problem and subproblem, as well as standard symmetry-breaking constraints included in the master problem to break symmetry in the solution space of first-stage decisions.

\subsection{Test Instances and Experimental Setup} \label{subsec:expt_setting}

We use three years of surgery duration data for six different surgery types from \cite{Mannino_et_al:2010} and parameter settings from the literature to generate various DRORSP instances as follows. First, we sample $500$ data points to estimate the mean, MAD, and lower and upper bounds of the surgery durations for each surgery type. In particular, we set the lower and upper bounds as the ($20$th, $80$th) or ($10$th, $90$th) percentiles of the empirical distribution. Second, for each combination of $|I|\in\{20, 21, \dots,25\}$ and $|R|\in\{7,10\}$, we generate the number of surgeries for each type from a multinomial distribution with probability being equal to the estimated surgery type proportion from the data set. Third, we set $T=480$ minutes and consider two sets of weights for the multi-criteria objective function $(\cf,\cv)\in\{(1,1/30),(1,1/120)\}$ as in \cite{Denton_et_al:2010}. Finally, for each combination of $|I|$, $|R|$, percentiles, and $(\cf,\cv)$, we generate and solve $5$ instances for a total of $240$ instances.

For algorithmic parameters, we set the final relative gap $\varepsilon\in\{2\%,5\%\}$ for both the C\&CG and i-C\&CG methods. In the i-C\&CG method, we set the initial master relative gap tolerance $\varepsilon_{MP}$ to $2\%$ and the inexact relative gap $\epst$ to $1.5\%$. In addition, we impose an initial solver time limit of $300$ seconds (s) for solving the master problems. If the time limit is exceeded, then we set $U^j$ as the best objective value found. At each exploitation step, $\varepsilon_{MP}^j$ is decreased by a factor of $\alpha=0.8$ and the time limit is increased by $600$ seconds. We implemented both the C\&CG and i-C\&CG methods with the AMPL modeling language and use CPLEX (version 20.1.0.0) as the solver with its default settings. We conducted all the experiments on a computer with an Intel Xeon Silver processor with a 2.10 GHz CPU and 128 Gb memory.

\subsection{Experimental Results} \label{subsec:expt_results}

We focus our discussion on the cost structure $(\cf,\cv)=(1,1/30)$; the results for $(\cf,\cv)=(1,1/120)$ are similar (see \ref{appdx:additional_results}). Figure \ref{fig:time_cost1} shows the time performance profiles for the C\&CG and i-C\&CG methods under two different lower and upper bound estimates. The curves represent the percentage of instances solved to a relative gap of $\varepsilon=2\%$ (dotted line) or $\varepsilon=5\%$ (solid line) within a given time $t\in[0,7200]$. Figure \ref{fig:time_cost1} clearly illustrates that the i-C\&CG method can solve more instances than the C\&CG method. For example, when we use (20th, 80th) percentiles, the C\&CG method can only solve \blue{around $70\%$} of the instances, but the i-C\&CG method can solve up to $90\%$ of them with $\varepsilon=5\%$ within the $2$-hour time limit. Similarly, when we use (10th, 90th) percentiles, the C\&CG method can solve only $20\%$ of the instances, but the i-C\&CG method can solve at least $50\%$ of them.
\begin{figure}[t]
    \centering
    \includegraphics[scale=0.72]{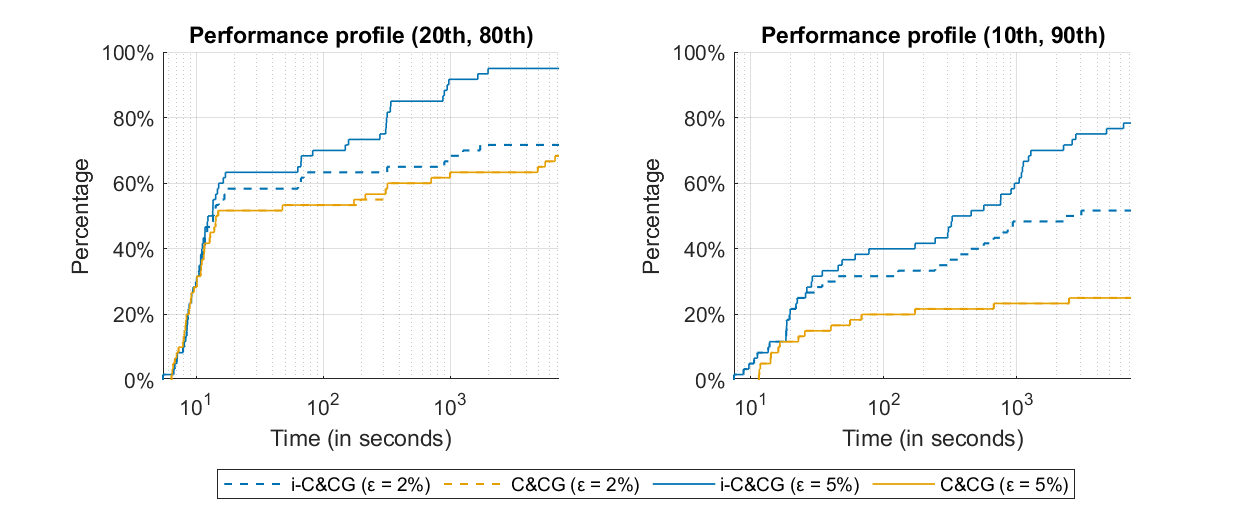}
    \caption{Time performance profile with $(\cf,\cv)=(1,1/30)$ under two different lower and upper bound estimates: left -- (20th, 80th) percentiles; right -- (10th, 90th) percentiles}
    \label{fig:time_cost1}
\end{figure}

Next, we analyze the final relative gap reported from the algorithms for instances that terminate at the $2$-hour time limit. The curves in Figure \ref{fig:gap_cost1} show the proportion of instances solved to a final relative gap less than a certain percentage when $\varepsilon=2\%$. It is clear that the final relative gaps from the i-C\&CG method are significantly less than those from the C\&CG method. For most instances, the final relative gaps from the i-C\&CG method are less than  $10\%$. In contrast,  the final relative gaps from the C\&CG method are greater than $80\%$ for at least $70\%$ of the instances.

We attribute the differences in performance of the two algorithms to the following. Since the master problem of the DRORSP is a challenging MILP, solving the master problem requires significant computational effort. By allowing inexact solutions to challenging master problems, the i-C\&CG method can solve problem instances more efficiently to a small relative gap. In contrast, we observe that the C\&CG method spends a significant amount of time solving some master problems (in early iterations), thus terminating with a large final relative gap for most instances.
\begin{figure}[t]
    \centering
    \includegraphics[scale=0.72]{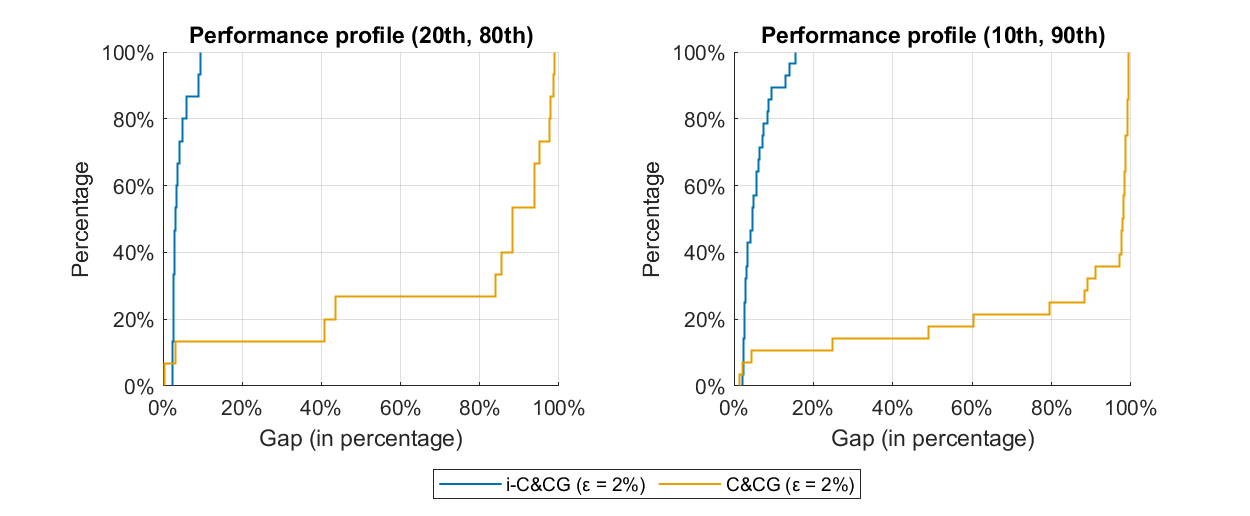}
    \caption{Gap performance profile for instances that both C\&CG and i-C\&CG exceed $2$-hour time limit with $(\cf,\cv)=(1,1/30)$ under two different lower and upper bound estimates: left -- (20th, 80th) percentiles; right -- (10th, 90th) percentiles}
    \label{fig:gap_cost1}
\end{figure}

\vspace{1mm}

\color{black}
\noindent \textbf{Acknowledgment}

\noindent We thank all colleagues who have contributed significantly to the related literature. In addition, we thank the editor and the two  anonymous reviewers for their insightful comments and suggestions. Dr. Karmel S.~Shehadeh dedicates her effort in this paper to every little dreamer in the whole world who has a dream so big and so exciting. Believe in your dreams and do whatever it takes to achieve them--the best is yet to come for you. \vspace{3mm}
\color{black}

\bibliographystyle{elsarticle-harv}
\bibliography{references}

\appendix
\newpage
\let\oldthebibliography=\thebibliography
\renewenvironment{thebibliography}[1]{%
   \oldthebibliography{#1}%
   \setcounter{enumiv}{1}%
}

\begin{center}
    {\Large An inexact column-and-constraint generation method to solve two-stage robust optimization problems (Appendices) \vspace{3mm}\\
    \large Man Yiu Tsang, Karmel S.~Shehadeh, Frank E.~Curtis}
\end{center}

\section{Two-stage Distributionally Robust Optimization Examples} \label{appdx:DRO}

By defining an ambiguity set $\calP$ (a subset of all probability measures with support $\Xi$), the two-stage distributionally robust optimization problem has the form
\begin{equation} \label{prob:2S_dist_robust}
    \min_{\xb\in\calX}\,\,\bigg\{ \cb^\tp \xb + \sup_{\Prob\in\calP} \, \E_{\Prob} \bigg[ \min_{\yb\in\calY(\xb,\xib)} \qb^\tp \yb  \bigg] \bigg\}.
\end{equation}
In the following, we provide two examples that under specific choices of the ambiguity set $\calP$ and the support $\Xi$, problem \eqref{prob:2S_dist_robust} can be reformulated in the form of \eqref{prob:2S_robust}. Thus, one can apply the C\&CG and i-C\&CG methods to solve these models.

\begin{example} \label{eg:mean_support_ambig}
Mean-support ambiguity set $\calP(\mub)=\{\Prob\in\calP(\Xi)\mid \E_\Prob(\xib)=\mub\}$ captures the support $\Xi$ and mean $\mub$ of the random vector $\xib$. Due to its intuitive inputs, this ambiguity set has been employed in various applications; see, e.g., \citepec{Jiang_et_al:2017ec, Shehadeh_Sanci:2021ec}. Under the Slater-type condition \citepec{Xu_et_al:2018ec}, the worst-case expectation in \eqref{prob:2S_dist_robust} is equivalent to its dual:
\begin{equation} \label{eg_eqn:mean_support_dual}
    \min_{\lambda\in\R^l}\,\, \bigg\{\mub^\tp\lambdab + \sup_{\xib\in\Xi}\bigg\{ Q(\xb,\xib)-\xib^\tp\lambdab \bigg\} \bigg\}.
\end{equation}
We remark that the Slater-type condition is weaker than the classical Slater condition; see, e.g., discussions in \citepec{Liu_et_al:2019ec}. The former is commonly used to ensure strong duality in distributionally robust optimization problems with moment-based ambiguity sets; see, e.g., \citepec{Xu_et_al:2018ec}. Hence, problem \eqref{eg_eqn:mean_support_dual} reduces to the form of problem \eqref{prob:2S_robust} with $(\xb,\lambdab)$ as the first-stage decision. 
\end{example}

\begin{example} \label{eg:MAD_ambig}
The mean-absolute-deviation ambiguity set $\calP(\mub,\sigmab)=\{\Prob\in\calP(\Xi)\mid \E_\Prob(\xib)=\mub,\, \E_\Prob|\xib-\mub|\leq \sigmab\}$ captures the support $\Xi$ and mean $\mub$ information, and requires the absolute deviation from mean is no more than $\sigmab$ \citepec{Wang_et_al:2019ec,Zhang_et_al:2017ec}. Again, under the Slater-type condition \citepec{Xu_et_al:2018ec}, we can reformulate the worst-case expectation in \eqref{prob:2S_dist_robust} as
\begin{equation} \label{eg_eqn:MAD_dual}
    \min_{\lambda\in\R^l,\,\rhob\in\R^l}\,\, \bigg\{\mub^\tp\lambdab + \rhob^\tp\sigmab + \sup_{\xib\in\Xi}\bigg\{ Q(\xb,\xib)-\xib^\tp\lambdab-\rhob^\tp |\xib-\mub| \bigg\} \bigg\},
\end{equation}
where the absolute value is computed entry-wisely. We can introduce auxiliary variables to linearize terms in absolute value, leading to the form of problem \eqref{prob:2S_robust}. 
\end{example}

\section{Proof of Proposition \ref{prop:valid_LB_iC&CG}} \label{appdx:proof_valid_LB}

To prove the validity of $L^\ell$ in Proposition \ref{prop:valid_LB_iC&CG}, we first provide Proposition \ref{prop:opt_val_MP_iC&CG}, which addresses the inexactness associated with the additional constraint \eqref{eqn:Inexact_C&CG_Master_con2}.

\begin{prop}   \label{prop:opt_val_MP_iC&CG}
If the value $\Lbar$ in \eqref{eqn:Inexact_C&CG_Master_con2} is greater than the optimal value $\upsilon^*$ of problem \eqref{prob:2S_robust} at some iteration $j$ (i.e., $\Lbar>\upsilon^*$), the optimal value $\upsilon^\star_j$ of \eqref{prob:Inexact_C&CG_Master} is $\Lbar$.
\end{prop}

\begin{proof}
First, note that constraint \eqref{eqn:Inexact_C&CG_Master_con2} implies $\upsilon_{j}^\star \geq \Lbar$. For the opposite direction that $\upsilon_{j}^\star \leq \Lbar$, it suffices to show that there exists a feasible solution with objective $\Lbar$. We claim that $(\xb,\delta)=(\xb^\star,\Lbar-\cb^\tp\xb^\star)$ is feasible to \eqref{prob:Inexact_C&CG_Master} with objective value $\Lbar$, where $\xb^\star$ is an optimal solution to problem \eqref{prob:2S_robust}. To verify our claim, we only need to show that \eqref{eqn:Inexact_C&CG_Master_con1} is satisfied. Indeed, since $\upsilon^\star < \Lbar$ by our assumption, we have
\begin{equation}  \label{pf_eqn:opt_val_MP_iC&CG_1}
 c^\tp x^\star + \max_{\xib\in\Xi} Q(\xb^\star,\xib)=\upsilon^\star < \Lbar.
\end{equation}
Therefore, we obtain the desired inequality
$$\max_{\xib\in\calS} Q(\xb^\star,\xib) \leq \max_{\xib\in\Xi} Q(\xb^\star,\xib) < \Lbar-\cb^\tp\xb^\star=\delta,$$
where the first inequality follows from $\calS\subseteq \Xi$ and the second one follows from \eqref{pf_eqn:opt_val_MP_iC&CG_1}. 
\end{proof}

\begin{proof}[Proof of Proposition \ref{prop:valid_LB_iC&CG}]
Suppose, on the contrary, that $L^\ell$ is not a valid lower bound for $\upsilon^*$, i.e., $L^\ell > \upsilon^*$. Consider the master problem \eqref{prob:Inexact_C&CG_Master} at iteration $\ell$ and its optimal value $\upsilon^\star_\ell$. Therefore, $\Lbar$ in this proof is the right-hand-side value of \eqref{eqn:Inexact_C&CG_Master_con2} when solving this master problem. First, we have $L^\ell> \Lbar$ from the only updating criterion for $\ell$ in step 1.2. Next, consider the following two cases. If $\Lbar > \upsilon^\star$, by Proposition~\ref{prop:opt_val_MP_iC&CG}, we have the inequality $\upsilon^\star_\ell = \Lbar < L^\ell$, contradicting that $L^\ell$ is a lower bound on $\upsilon_\ell^\star$. If $\Lbar \leq \upsilon^\star$, the solution $(\xb,\delta)=(\xb^\star,\max_{\xib\in\Xi} Q(\xb^\star,\xib))$ is feasible to the master problem with objective  $\upsilon^\star$, where $\xb^\star$ is an optimal solution to problem \eqref{prob:2S_robust}. Therefore, this leads to the same contradiction that $\upsilon^\star_{\ell} \leq \upsilon^\star<L^\ell$. This concludes that $L^\ell$ is a valid lower bound. 
\end{proof}

\section{Proof of Corollary \ref{coro:terminate_MIP_gap_iC&CG}} \label{appdx:proof_coro_terminate_MIP_gap_iC&CG}

\begin{proof}
Note that the new updating rule $\Lbar\leftarrow U^j$ guarantees that $\Lbar$ is always a valid lower bound on $\upsilon^\star$. Therefore, we can set $\ell=j$ at each iteration and the desired results follow from the first part of the proof in Proposition \ref{prop:terminate_MIP_gap_iC&CG}.
\end{proof}

\section{Pseudocode for i-C\&CG variants} \label{appdx:iCCG_variant}

We present the pseudo codes for the two variant of i-C\&CG discussed in Section \ref{subsec:inexact_C&CG_variant} in the following two subsections, respectively. For the ease of reading, we highlight additional elements in these two variants, when compared with the general i-C\&CG framework presented in Algorithm \ref{algo:inexact_C&CG}, in orange.

\subsection{Exploitation Frequency Variant}

Algorithm \ref{algo:inexact_C&CG_V1} presents the exploitation frequency variant of the i-C\&CG method. In this variant, we have an additional parameter $f^\text{exploit}$ to control the exploitation frequency. Specifically, in the backtracking routine, if $j-\ell > f^\text{exploit}$ (i.e., exploitation has not been reached at least $f^\text{exploit}$ times), the algorithm proceeds to the exploitation step. In such a case ($\ell \ll j$), the scenario set is substantially enlarged, and thus, it is likely that a new valid lower bound could be found via exploitation.
\IncMargin{1.3em}
\begin{algorithm}[h] 
\setstretch{1.1}
\SetKwInOut{Initialization}{Initialization}
\Initialization{$\Lbar\gets0$, $\Ubar\gets\infty$, $\varepsilon \in [0,1]$, $\epst\in(0,\varepsilon/(1+\varepsilon))$,  $\{\varepsilon_{MP}^j \in (0,1)\}_{j\in\N}$, $\alpha \in (0,1)$, $\calS \gets\emptyset$, $j\gets1$, $\ell\gets0$, \textcolor{orange}{$f^\text{exploit}\in\mathbb{N}$}.} 
\textbf{1. Master problem.} \\
\textbf{\small\hspace{5.3mm}1.1.} Solve the master problem \eqref{prob:Inexact_C&CG_Master} to within a relative optimality gap of $\varepsilon_{MP}^j$. \\ \hspace{5.3mm}Record the best feasible solution $(\xb^j,\delta^j)$ found.  \\
\textbf{\small\hspace{5.3mm}1.2.} Record a lower bound $L^j\geq \Lbar$ and upper bound $U^j= \cb^\tp\xb^j + \delta^j$ of $\upsilon^\star_j$. \\
\hspace{5.3mm}If $L^j> \Lbar$, then set $\ell \leftarrow  j$.  \\
\textbf{\small\hspace{5.3mm}1.3.} Set $\Lbar \leftarrow  U^j$.\\
\textbf{2. Subproblem.} Solve the subproblem \eqref{prob:Exact_C&CG_Subproblem} for fixed $\xb=\xb^j$. \\
\hspace{5.3mm}Record the optimal solution $\xib^\star$ and value $D^j$. \\
\hspace{5.3mm}Set $\Ubar \leftarrow  \min\big\{\Ubar,\, \cb^\tp\xb^j +D^j\big\}$.\\

\textbf{3. Optimality test and backtracking routine.} \\
\hspace{5.3mm}If $(\Ubar-L^\ell)/\Ubar<\varepsilon$, then terminate and return $\xb^j$; otherwise, do the following.
\begin{itemize}
    \item \textbf{Exploitation}: If $(\Ubar-U^j)/\Ubar < \epst$ \textcolor{orange}{or $j-\ell > f^\text{exploit}$}, then set $j\leftarrow \ell$ and $\Lbar\leftarrow L^\ell$. \\
    Set $\varepsilon^j_{MP}\leftarrow \alpha \varepsilon^j_{MP}$ for all $j\geq \ell$ and go back to step 1. \vspace{-3mm}
    \item \textbf{Exploration}: If $(\Ubar-U^j)/\Ubar  \geq \epst$, then go to step 4.  \vspace{-3mm}
\end{itemize}
\textbf{4. Scenario set enlargement.}\\
\hspace{5.3mm}Enlarge the scenario set $\calS \leftarrow \calS\cup\{\xib^\star\}$. \\ \hspace{5.3mm}Update $j \leftarrow j+1$ and go back to step 1.
\BlankLine
\caption{Inexact column-and-constraint (i-C\&CG) method, exploitation frequency variant} \label{algo:inexact_C&CG_V1}
\end{algorithm}\DecMargin{1em}

\subsection{Time Limit Variant}

Algorithm \ref{algo:inexact_C&CG_V2} presents the time limit variant of the i-C\&CG method. In this variant, there are two additional parameters, $\tau$ and $\beta$, to control the time limit of the master problem and the increase for the time limit, respectively. Specifically, in step 1.1, we impose a time limit $\tau$ when solving the master problem. This is useful when the relative gap $\varepsilon_{MP}^j$ is difficult to achieve. To control the inaccuracies due to the imposed time limit, in the exploitation step, the algorithm increases the time limit $\tau$ by $\beta$. 
\IncMargin{1.3em}
\begin{algorithm}[t!] 
\setstretch{1.1}
\SetKwInOut{Initialization}{Initialization}
\Initialization{$\Lbar \gets0$, $\Ubar\gets\infty$, $\varepsilon \in[0,1]$, $\epst\in(0,\varepsilon/(1+\varepsilon))$,  $\{\varepsilon_{MP}^j \in (0,1)\}_{j\in\N}$, $\alpha \in (0,1)$, $\calS\gets\emptyset$, $j\gets1$, $\ell\gets0$, \textcolor{orange}{$\tau > 0$, $\beta>0$}.} 
\textbf{1. Master problem.} \\
\textbf{\small\hspace{5.3mm}1.1.} Solve the master problem \eqref{prob:Inexact_C&CG_Master} to within a relative optimality gap of $\varepsilon_{MP}^j$ \\\hspace{5.3mm}\textcolor{orange}{or terminate if the time limit $\tau$ is exceeded}.\\ \hspace{5.3mm}Record the best feasible solution $(\xb^j,\delta^j)$ found.  \\
\textbf{\small\hspace{5.3mm}1.2.} Record a lower bound $L^j\geq \Lbar$ and upper bound $U^j= \cb^\tp\xb^j + \delta^j$ of $\upsilon^\star_j$. \\ \hspace{5.3mm}If $L^j> \Lbar$, then set $\ell \leftarrow  j$.  \\
\textbf{\small\hspace{5.3mm}1.3.} Set $\Lbar \leftarrow  U^j$.\\
\textbf{2. Subproblem.} Solve the subproblem \eqref{prob:Exact_C&CG_Subproblem} for fixed $\xb=\xb^j$. \\
\hspace{5.3mm}Record the optimal solution $\xib^\star$ and value $D^j$. \\ \hspace{5.3mm}Set $\Ubar \leftarrow  \min\big\{\Ubar,\, \cb^\tp\xb^j +D^j\big\}$.\\

\textbf{3. Optimality test and backtracking routine.} \\
\hspace{5.3mm}If $(\Ubar-L^\ell)/\Ubar<\varepsilon$, then terminate and return $\xb^j$; otherwise, do the following.
\begin{itemize}
    \item \textbf{Exploitation}: If $(\Ubar-U^j)/\Ubar < \epst$, then set $j\leftarrow \ell$ and $\Lbar\leftarrow L^\ell$. \\ Set $\varepsilon^j_{MP}\leftarrow \alpha \varepsilon^j_{MP}$ for all $j\geq \ell$, \textcolor{orange}{set $\tau\leftarrow \tau+\beta$}, and go back to step 1. \vspace{-3mm}
    \item \textbf{Exploration}: If $(\Ubar-U^j)/\Ubar  \geq \epst$, then go to step 4.  \vspace{-3mm}
\end{itemize}
\textbf{4. Scenario set enlargement.}\\
\hspace{5.3mm}Enlarge the scenario set $\calS \leftarrow \calS\cup\{\xib^\star\}$. \\ \hspace{5.3mm}Update $j \leftarrow j+1$ and go back to step 1.
\BlankLine
\caption{Inexact column-and-constraint (i-C\&CG) method, time limit variant} \label{algo:inexact_C&CG_V2}
\end{algorithm}\DecMargin{1em}


\section{Details of the DRORSP} \label{appdx:DRORSP}

\subsection{Master-Subproblem Framework}  \label{appdx:DRORSP_MS_framework}
In this section, we derive the master problem and subproblem for the DRORSP that facilitates the use of the C\&CG and i-C\&CG methods. In view of \eqref{prob:DRORSP}, we define the \blue{second-stage problem} as
\begin{subequations}
\begin{align}
Q(\yb,\db):=\sum_{r\in R}\Bigg(\sum_{i\in I} y_{ir}d_i - T\Bigg)^+ =\underset{\wb}{\text{minimize}}\, \quad
&  \sum_{r\in R} w_r\\
\text{subject to} \quad
& w_r \geq \sum_{i\in I}y_{ir}d_i - T,\quad\forall r\in R, \\
& w_r \geq 0,\quad\forall r\in R.
\end{align} \label{prob:DRORSP_SS}%
\end{subequations}
We first reformulate the inner maximization problem in \eqref{prob:DRORSP} over $\Prob\in\calP$ as defined in \eqref{eqn:MAD_ambiguity_set}. As shown in \citeec{Wang_et_al:2019ec}, this inner maximization problem is equivalent to its dual presented in Proposition \ref{prop:DRORSP_inner_max_dual_reformulation}.

\begin{prop} \label{prop:DRORSP_inner_max_dual_reformulation}
The problem $\sup_{\Q\in\calP} \E_\Q[Q(\yb,\db)]$ with $\calP$ defined in \eqref{eqn:MAD_ambiguity_set} is equivalent to
\begin{subequations}
\begin{align}
\underset{\etab,\,\vphib}{\textup{minimize}}\, \quad
&  \sum_{i\in I} (\mu_i\eta_i + \sigma_i\varphi_i) + \sup_{\db\in\Xi}\bigg\{ Q(\yb,\db)- \sum_{i\in I}\Big( d_i\eta_i + |d_i-\mu_i| \varphi_i \Big) \bigg\} \\
\textup{subject to} \quad
& \varphi_i\geq 0,\quad\forall i\in I.
\end{align} \label{prob:DRORSP_inner_max_dual}%
\end{subequations}  \vspace{-10mm}
\end{prop}
\begin{proof}
First, note that $\Xi$ is compact. Moreover, $Q(\xb,\db)$ is a continuous function in $\db$, as well as $\phi^\text{eq}_{i}(\db) := d_i$ and $\phi^\text{ineq}_{i}(\db) := |d_i-\mu_i|$ for all $i\in I$. By strong duality of moment problems (see, e.g., \citeec{Shapiro_et_al:2014ec}), the problem $\sup_{\Q\in\calP} \E_\Q[Q(\yb,\db)]$ is equivalent to
\allowdisplaybreaks
\begin{subequations}
\begin{align}
\underset{\etab,\,\vphib,\,\theta}{\textup{minimize}}\, \quad
&  \sum_{i\in I} (\mu_i\eta_i + \sigma_i\varphi_i) + \theta \\
\textup{subject to} \quad
& \sum_{i\in I}\Big( d_i\eta_i + |d_i-\mu_i| \varphi_i \Big) + \theta \geq Q(\yb,\db) ,\quad\forall \db\in\Xi, \label{pf_eqn:DRORSP_inner_max_dual_reformulation_1} \\
& \varphi_i\geq 0,\quad\forall i\in I.
\end{align}%
\end{subequations}
From \eqref{pf_eqn:DRORSP_inner_max_dual_reformulation_1}, we have
$$\theta \geq Q(\yb,\db)- \sum_{i\in I}\Big( d_i\eta_i + |d_i-\mu_i| \varphi_i \Big),\quad\forall \db\in\Xi. $$
Since $\theta$ is unrestricted and the objective is to minimize $\theta$, this shows the equivalence between \eqref{pf_eqn:DRORSP_inner_max_dual_reformulation_1} and \eqref{prob:DRORSP_inner_max_dual}.
\end{proof}

In view of Proposition \ref{prop:DRORSP_inner_max_dual_reformulation}, the DRORSP \eqref{prob:DRORSP} is equivalent to
\begin{subequations}
\begin{align}
\underset{\xb,\,\yb,\,\etab,\,\vphib,\,\delta}{\text{minimize}}\, \quad
&  \sum_{r\in R} \cf x_r + \cv \sum_{i\in I} (\mu_i\eta_i + \sigma_i\varphi_i) + \cv\delta\\
\text{subject to} \quad
& y_{ir} \leq x_r,\quad\forall i\in I,\, r\in R, \\
& \sum_{r\in R}y_{ir} = 1,\quad\forall i\in I,\\
& \delta \geq Q(\yb,\db) - \sum_{i\in I}\Big( d_i\eta_i + |d_i-\mu_i| \varphi_i \Big),\quad\forall \db\in\Xi, \\
& \varphi_i \geq 0,\quad\forall i\in I, \\
& x_r\in\{0,1\},\,y_{ir}\in\{0,1\},\quad\forall i\in I,\, r\in R.
\end{align}%
\end{subequations}
Therefore, given a subset of scenario $\calS\subset\Xi$, the master problem is given by
\begin{subequations}
\begin{align}
\underset{\xb,\,\yb,\,\etab,\,\vphib,\,\delta,\,\wb}{\text{minimize}}\, \quad
&  \sum_{r\in R} \cf x_r + \cv \sum_{i\in I} (\mu_i\eta_i + \sigma_i\varphi_i) + \cv\delta\\
\text{subject to} \quad
& y_{ir} \leq x_r,\quad\forall i\in I,\, r\in R, \\
& \sum_{r\in R}y_{ir} = 1,\quad\forall i\in I,\\
& \delta \geq \sum_{r\in R} w_r^k - \sum_{i\in I} ( d^k_i\eta_i + |d^k_i-\mu_i|\varphi_i ), \quad\forall k\in\calS, \\
& w_r^k \geq \sum_{i\in I}y_{ir}d^k_i - T,\quad\forall r\in R,\, k\in\calS, \\
& \varphi_i\geq 0,\, w_r^k \geq 0,\quad\forall i\in I,\, r\in R,\, k\in\calS,\\
& x_r\in\{0,1\},\,y_{ir}\in\{0,1\},\quad\forall i\in I,\, r\in R.
\end{align}%
\end{subequations}
Finally, we provide a tractable MILP reformulation of the subproblem in Proposition \ref{prop:DRORSP_subproblem_reformulation}.

\begin{prop} \label{prop:DRORSP_subproblem_reformulation}
Let $\Deltalb_i = \mu_i-\dlb_i$ and $\Deltaub_i=\dub_i-\mu_i$ for all $i\in I$. The subproblem $\sup_{\db\in\Xi}\Big\{ Q(\yb,\db)- \sum_{i\in I}\big( d_i\eta_i + |d_i-\mu_i| \varphi_i \big) \Big\}$ is equivalent to
\allowdisplaybreaks
\begin{subequations}
\begin{align}
\underset{\pib,\,\bb,\,\zetab}{\textup{maximize}}\, \quad
& -T\sum_{r\in R}\pi_r + \sum_{i\in I} \Bigg[\mu_i\Bigg(\sum_{r\in R}\pi_r y_{ir} - \eta_i \Bigg) - \Deltalb_i \Bigg(\sum_{r\in R}\zeta^1_{ir} y_{ir} -\eta_i b^1_i \Bigg) + \Deltaub_i \Bigg(\sum_{r\in R} \zeta^2_{ir} y_{ir} - \eta_i b^2_i \Bigg) \nonumber \\
& \quad -\varphi_i\Big(b^1_i\Deltalb_i + b^2_i \Deltaub_i\Big)  \Bigg]\\
\textup{subject to} \quad
& 0\leq \pi_r\leq 1,\quad\forall r\in R,\\
& \zeta^1_{ir}\geq 0,\,\,\zeta^1_{ir}\geq \pi_r +b^1_i-1,\,\,\zeta^1_{ir}\leq \pi_r,\,\,\zeta^1_{ir}\leq b^1_r,\quad\forall i\in I,\, r\in R, \\
& \zeta^2_{ir}\geq 0,\,\,\zeta^2_{ir}\geq \pi_r +b^2_i-1,\,\,\zeta^2_{ir}\leq \pi_r,\,\,\zeta^2_{ir}\leq b^2_r,\quad\forall i\in I,\, r\in R, \\
& b^1_i + b^2_i \leq 1,\quad\forall i\in I,\\
& b^1_i\in\{0,1\},\, b^2_i\in\{0,1\},\quad\forall i\in I.
\end{align} \label{pf_prob:DRORSP_subproblem_MILP}%
\end{subequations} 
\end{prop} \vspace{-10mm}
\begin{proof}
First, by LP strong duality, we have
\begin{subequations}
\begin{align}
Q(\yb,\db)=\underset{\pib}{\text{maximize}}\, \quad
&  \sum_{r\in R} \pi_r\Bigg(\sum_{i\in I} y_{ir}d_i - T\Bigg)\\
\text{subject to} \quad
& 0\leq \pi_r \leq 1,\quad\forall r\in R.
\end{align} \label{pf_prob:DRORSP_SS_dual}%
\end{subequations}
Therefore, we can reformulate the subproblem as a single maximization problem:
\begin{subequations}
\begin{align}
\underset{\pib,\,\db}{\text{maximize}}\, \quad
& -T\sum_{r\in R}\pi_r + \sum_{i\in I} \Bigg[\Bigg(\sum_{r\in R}\pi_r y_{ir} -\eta_i \Bigg) d_i -\varphi_i|d_i-\mu_i| \Bigg]\\
\text{subject to} \quad
& 0\leq \pi_r\leq 1,\quad\forall r\in R,\\
& \dlb_i \leq d_i \leq \dub_i,\quad\forall i\in I.
\end{align} \label{pf_prob:DRORSP_subproblem}%
\end{subequations}
Since the objective function of \eqref{pf_prob:DRORSP_subproblem} is piecewise linear in $d_i$ with two pieces on $[\dlb_i,\mu_i]$ and $[\mu_i,\dub_i]$, we have that the optimal solution $d_i^\star\in\{\dlb_i,\mu_i,\dub_i\}$. Let $b^1_i$ and $b^2_i$ be two binary variables and let $d_i=\mu_i - b^1_i\Deltalb_i + b^2_i\Deltaub_i$. Then, problem \eqref{pf_prob:DRORSP_subproblem} is equivalent to
\begin{subequations}
\begin{align}
\underset{\pib,\,\bb}{\text{maximize}}\, \quad
& -T\sum_{r\in R}\pi_r + \sum_{i\in I} \Bigg[\mu_i\Bigg(\sum_{r\in R}\pi_r y_{ir} - \eta_i \Bigg) - \Deltalb_i \Bigg(\sum_{r\in R}\pi_r b^1_i y_{ir} -\eta_i b^1_i \Bigg) \nonumber \\
& \quad  + \Deltaub_i \Bigg(\sum_{r\in R} \pi_r b^2_i y_{ir} - \eta_i b^2_i \Bigg)  -\varphi_i\Big(b^1_i\Deltalb_i + b^2_i \Deltaub_i\Big)  \Bigg]\\
\text{subject to} \quad
& 0\leq \pi_r\leq 1,\quad\forall r\in R,\\
& b^1_i + b^2_i \leq 1,\quad\forall i\in I,\\
& b^1_i\in\{0,1\},\, b^2_i\in\{0,1\},\quad\forall i\in I.
\end{align} \label{pf_prob:DRORSP_subproblem2}%
\end{subequations}
Note that \eqref{pf_prob:DRORSP_subproblem2} is non-linear due to the quadratic terms $b^1_i\pi_r$ and $b^2_i\pi_r$ in the objective. Defining $\zeta^1_{ir}=b^1_i\pi_r$ and $\zeta^2_{ir}=b^2_i\pi_r$ and introducing the McCormick inequalities, we can reformulate \eqref{pf_prob:DRORSP_subproblem2} into \eqref{pf_prob:DRORSP_subproblem_MILP}.
\end{proof}

\subsection{Symmetry-Breaking Constraints}  \label{appdx:DRORSP_SBC}

As in \citeec{Wang_et_al:2019ec}, we apply the following symmetry-breaking constraints \citeec{Denton_et_al:2010ec} in the model under the realistic assumption that $|I|\geq |R|$:
\begin{equation} \label{eqn:SBC1}
    \blue{x_r \geq x_{r+1},\quad\forall r\in\{1,\dots, |R|-1\},}
\end{equation}
\begin{equation} \label{eqn:SBC2}
    \sum_{r=1}^i y_{ir} = 1,\quad\forall i\in\{1,\dots, |R|-1\},
\end{equation}
\begin{equation} \label{eqn:SBC3}
    \sum_{r=j}^{\min\{i,|R|\}} y_{ir} \leq \sum_{u=j-1}^{i-1} y_{u,j-1}\quad\forall j\in\{2,\dots,|R|\},\, i\in\{j,\dots,|R|\}.
\end{equation}
\blue{Constraints \eqref{eqn:SBC1} require that ORs with smaller indices are open before those with larger indices. In other words, ORs are open in ascending order of the OR index $r$: OR1 is open before OR2 is open, OR2 is open before OR3, and so on. Constraints \eqref{eqn:SBC2} require that surgery $i$ is assigned to one of the ORs with index $r\in\{1,\dots,i\}$. Finally, constraints \eqref{eqn:SBC3} ensure that if surgery $i$ is assigned to an OR with index $r\in\big\{j,\dots,\min\{i,|R|\}\big\}$, then there exists at least one surgery with index $u\in\{j-1,\dots,i-1\}$ that is assigned to OR with index $j-1$. We refer readers to \citeec{Denton_et_al:2010ec} for detailed explanations and examples. }


\section{Additional Computational Results} \label{appdx:additional_results}

We provide additional computational results under cost structure $(\cf,\cv)=(1,1/120)$. Figure \ref{fig:time_cost2} shows the time performance profile under two different lower and upper bound estimates. Similar to the observations in Section \ref{subsec:expt_results}, while the computational performance between the C\&CG and i-C\&CG methods for easier instances is similar, we observe that the i-C\&CG method is more efficient in solving challenging instances. When we use (20th, 80th) percentiles, the C\&CG method can only solve about $40\%$ of the instances, but the i-C\&CG method can solve more than $80\%$ of them (indeed, all the instances when $\varepsilon=5\%$) within the $2$-hour time limit. When we use (10th, 90th) percentiles, the C\&CG method can solve only less than $10\%$ of the instances while the i-C\&CG method can solve up to $90\%$ with $\varepsilon=5\%$.
\begin{figure}[t]
    \centering
    \includegraphics[scale=0.72]{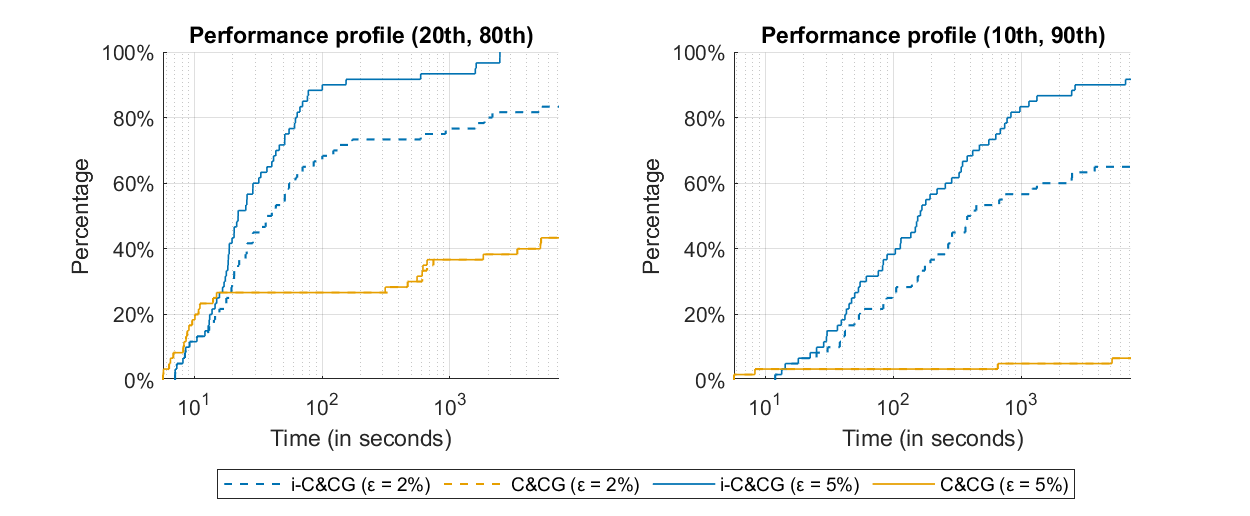}
    \caption{Time performance profile with $(\cf,\cv)=(1,1/120)$ under two different lower and upper bound estimates: left -- (20th, 80th) percentiles; right -- (10th, 90th) percentiles}
    \label{fig:time_cost2}
\end{figure}

Finally, we also compare the final relative gap of instances that both the C\&CG and i-C\&CG methods terminate due to the time limit. Figure \ref{fig:gap_cost2} shows the corresponding results when $(\cf,\cv)=(1,1/120)$ and $\varepsilon=2\%$. Similar to the observations in Section \ref{subsec:expt_results}, the final relative gaps from the i-C\&CG method are smaller than that of the C\&CG method. For instance, when we use (20th, 80th) percentiles, all the final relative gaps are within $10\%$, but those from the C\&CG method are still greater than $20\%$ (and many of them are even greater than $80\%$). We observe similar results when using (10th, 90th) percentiles. These results further conclude that our proposed i-C\&CG method could be efficient when solving challenging instances (i.e., with difficult master problems).
\begin{figure}[t!]
    \centering
    \includegraphics[scale=0.72]{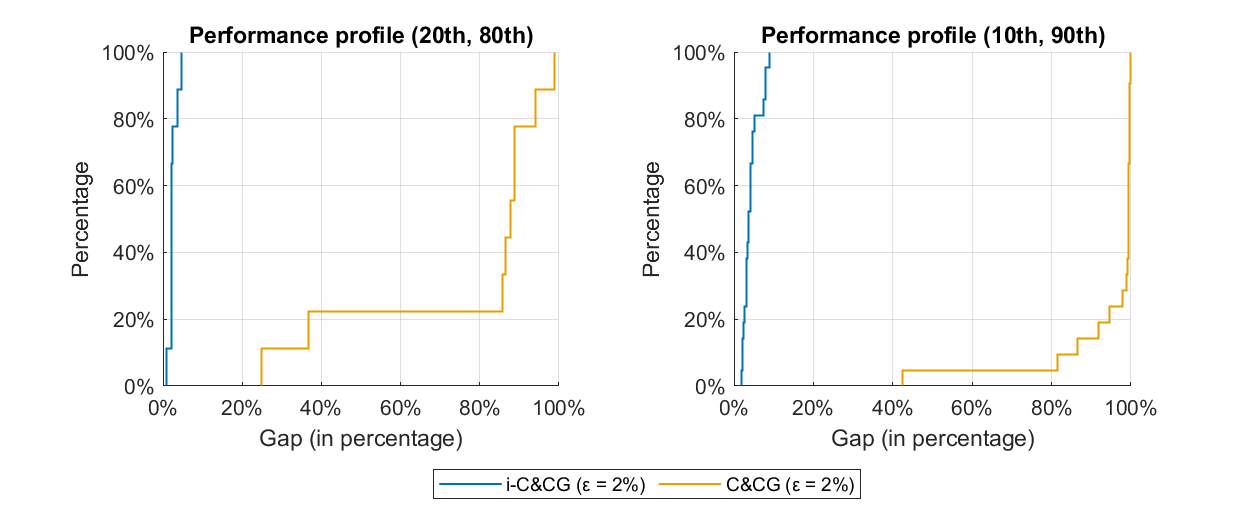}
    \caption{Gap performance profile for instances that both the C\&CG and i-C\&CG methods exceed the $2$-hour time limit with $(\cf,\cv)=(1,1/120)$ under two different lower and upper bound estimates: left -- (20th, 80th) percentiles; right -- (10th, 90th) percentiles}
    \label{fig:gap_cost2}
\end{figure}

\section{Additional Experiments on a Facility Location Problem} \label{appdx:additional_results_ROCPCP}
\blue{
In this section, we present additional computational results using a robust facility location problem. In \ref{appdx:ROCPCP}, we describe the problem setting. In \ref{appdx:ROCPC_MS_framework}, we derive the master problem and subproblem that enables us to employ the C\&CG and i-C\&CG methods. In \ref{appdx:ROCPC_comp_results}, we present numerical results comparing the computational performance of the C\&CG and i-C\&CG methods. Finally, in \ref{appdx:ROCPC_iCCG_variants}, we present numerical examples demonstrating the potential benefits of the proposed i-C\&CG variants in Section \ref{subsec:inexact_C&CG_variant}.
}

\subsection{Robust Capacitated $p$-Center Problem (ROCPCP)} \label{appdx:ROCPCP}

\blue{
We start by introducing the ROCPCP setting. Let $I$ be a the set of customer locations and $J$ be a set of facilities. Each customer location $i\in I$ has a random demand $d_i$. We define demand vector as $\db=(d_1,\dots, d_{|I|})^\tp$. The cost of transporting one unit from facility $j\in J$ to customer location $i\in I$ is $c_{ij}$, and the capacity of facility $j\in J$ is $C_j$. Given the sets of customer locations $I$ and facilities $J$, in the ROCPCP, we want to make the following decisions: (a) decide which facility to open, and (b) assign each customer to an open facility. The number of facilities should be at most $p$. To model demand uncertainty, we define the budgeted uncertainty set as (see \citeec{Zeng_Zhao:2013ec})
\begin{equation} \label{eqn:budget_uncertainty_set}
    \calD=\Bigg\{ \db \,\Bigg|\,d_i = \mu_i + b_i\Delta_i,\, \sum_{i\in I} b_i \leq \tau,\, b_i\in\{0,1\},\, i\in I\Bigg\},
\end{equation}
where $\mu_i$ is the nominal demand, $\Delta_i$ is the maximal deviation from the nominal demand, and $\tau$ is an integer controlling the number of demand deviations from the nominal value.
}

\blue{For each $j\in J$, we define a binary variable $y_j$ that equals $1$ if facility $j$ is open, and is $0$ otherwise. In addition, we define a binary variable $x_{ij}$ that equals $1$ if customer location $i$ is assigned to facility $j$, and is $0$ otherwise. The ROCPCP can now be stated as follows:}

\blue{
\begin{subequations}
\begin{align}
\underset{\yb\in\{0,1\}^{|J|},\,\xb\in\{0,1\}^{|I|\times|J|}}{\text{minimize}}\, \quad
& \sup_{\db\in\calD} \Bigg\{\max_{i\in I} \sum_{j\in J} c_{ij}d_i x_{ij} \Bigg\}\label{eqn:ROCPCP_obj}\\
\text{subject to} \hspace{13mm}
& \sum_{j\in J} x_{ij}=1,\quad\forall i\in I, \label{eqn:ROCPCP_con1}\\
& x_{ij}\leq y_j,\quad\forall i\in I,\, j\in J, \label{eqn:ROCPCP_con2}\\
& \sum_{j\in J} y_j \leq p, \label{eqn:ROCPCP_con3}\\
& \sum_{i\in I} (\mu_i+\Delta_i) x_{ij} \leq C_j ,\quad\forall j\in J, \label{eqn:ROCPCP_con4}
\end{align} \label{prob:ROCPCP}%
\end{subequations}
Objective \eqref{eqn:ROCPCP_obj} is the worst-case maximum transportation cost over different customer location $i\in I$ and demand realization $\db\in\calD$. Constraints \eqref{eqn:ROCPCP_con1} require that each customer location is assigned to exactly one facility, and constraints \eqref{eqn:ROCPCP_con2} ensure that customers are only assigned to open facilities. Constraint \eqref{eqn:ROCPCP_con3} ensures that the number of open facilities is at most $p$.  Constraints \eqref{eqn:ROCPCP_con4} require that an open facility can fulfill the demand of the assigned customer to it in the worst-case scenario, i.e., when demand $d_i$ takes its upper value $\mu_i+\Delta_i$.}


\subsection{Master-Subproblem Framework}  \label{appdx:ROCPC_MS_framework}

\blue{
In view of \eqref{prob:ROCPCP}, given a subset of scenario $\calS\subset\calD$, the master problem is given by
\begin{subequations}
\begin{align}
\underset{\yb,\,\xb,\,\zb,\,\delta}{\text{minimize}}\, \quad
& \delta\\
\text{subject to} \quad\,
& \sum_{j\in J} x_{ij}=1,\quad\forall i\in I, \\
& x_{ij}\leq y_j,\quad\forall i\in I,\, j\in J, \\
& \sum_{j\in J} y_j \leq p, \\
& \sum_{i\in I} (\mu_i+\Delta_i) x_{ij} \leq C_j ,\quad\forall j\in J, \\
& \delta \geq z^k,\quad\forall k\in\calS, \\
& z^k \geq \sum_{j\in J} c_{ij}d_i^k x_{ij},\quad\forall i\in I,\,k\in\calS, \\
& y_j\in\{0,1\},\, x_{ij}\in\{0,1\},\quad\forall i\in I,\, j\in J.
\end{align}%
\end{subequations}
Next, we provide a tractable reformulation of the subproblem. First, we define the second-stage problem as 
\begin{subequations}
\begin{align}
Q(\xb,\db):=\underset{z}{\text{minimize}}\, \quad
& z\\
\text{subject to} \quad
& z \geq \sum_{j\in J} c_{ij} d_i x_{ij} ,\quad\forall i\in I.
\end{align} \label{prob:ROCPCP_SS}%
\end{subequations}
In Proposition \ref{prop:ROCPCP_inner_max_dual_reformulation}, we derive a tractable MILP reformulation of the subproblem $\sup_{\db\in\calD} Q(\xb,\db)$.
}

\blue{
\begin{prop} \label{prop:ROCPCP_inner_max_dual_reformulation}
The subproblem $\sup_{\db\in\calD} Q(\xb,\db)$ with $\calD$ defined in \eqref{eqn:budget_uncertainty_set} is equivalent to
\begin{subequations}
\begin{align}
\underset{\pib,\,\bb,\,\zetab}{\textup{minimize}}\, \quad
&  \sum_{i\in I}\sum_{j\in J} c_{ij}x_{ij}\mu_i\pi_i + \sum_{i\in I}\sum_{j\in J} c_{ij}x_{ij} \Delta_i\zeta_i \\
\textup{subject to} \quad
& \sum_{i\in I}\pi_i=1,\\
& \sum_{i\in I} b_i \leq \tau, \\
& \zeta_i\geq 0,\,\,\zeta_i\geq \pi_i +b_i-1,\,\,\zeta_i\leq \pi_i,\,\,\zeta_i\leq b_i,\quad\forall i\in I, \\
& \pi_i \geq 0,\, b_i\in\{0,1\},\quad\forall i\in I.
\end{align} \label{prob:ROCPCP_subproblem_MILP}%
\end{subequations}  \vspace{-10mm}
\end{prop}
}

\blue{
\begin{proof}
First, by LP strong duality, we have
\begin{subequations}
\begin{align}
Q(\xb,\db)=\underset{\pib}{\text{maximize}}\, \quad
&  \sum_{i\in I} \pi_i \Bigg( \sum_{j\in J} c_{ij}d_ix_{ij} \Bigg)\\
\text{subject to} \quad
& \sum_{i\in I}\pi_i=1, \\
& \pi_i \geq 0,\quad\forall i\in I.
\end{align} \label{pf_prob:ROCPCP_SS_dual}%
\end{subequations}
Therefore, using the definition of the uncertainty set $\calD$ in \eqref{eqn:budget_uncertainty_set}, we can reformulate the subproblem $\sup_{\db\in\calD} Q(\xb,\db)$ as a single maximization problem:
\begin{subequations}
\begin{align}
\underset{\pib,\,\bb}{\text{maximize}}\, \quad
& \sum_{i\in I} \pi_i \Bigg[ \sum_{j\in J} c_{ij}\big(\mu_i + b_i\Delta_i\big)x_{ij} \Bigg]\\
\text{subject to} \quad
& \sum_{i\in I}\pi_i=1,\\
& \sum_{i\in I} b_i \leq \tau, \\
& \pi_i \geq 0,\, b_i\in\{0,1\},\quad\forall i\in I.
\end{align} \label{pf_prob:ROCPCP_subproblem}%
\end{subequations}
Note that \eqref{pf_prob:ROCPCP_subproblem} is non-linear due to the quadratic term $b_i\pi_i$ in the objective. Defining $\zeta_{i}=b_i\pi_i$ and introducing the McCormick inequalities, we can reformulate \eqref{pf_prob:ROCPCP_subproblem} into \eqref{prob:ROCPCP_subproblem_MILP}.
\end{proof}
}

\subsection{Computational Results}  \label{appdx:ROCPC_comp_results}

\blue{
We follow the same parameter settings in \citeec{Zeng_Zhao:2013ec} to generate problem instances. For the uncertainty set parameters, we generate $\mu_i$ from $U[10,500]$ and $\alpha_i$ from $U[0.1,0.5]$, and then compute $\Delta_i=\alpha_i\mu_i$. Here, $U[a,b]$ denotes the uniform distribution on $[a,b]$. We generate the transportation cost $c_{ij}$ from $U[10,500]$ and capacity $C$ from $U[1000,1500]$, where the capacity is larger than those generated in \citeec{Zeng_Zhao:2013ec}.  Note that ROCPCP is a capacitated problem, implying that a generated instance could be infeasible if the capacities are too small. For illustrative purposes, we generate feasible instances. Finally, we set $|I|=100$, $p=|I|/4$ and $\tau\in\{0.2|I|, 0.5|I|, 0.8|I|\}$ (rounded up to the nearest integer) which corresponds to some challenging instances. For each $\tau$, we generate and solve $5$ instances for a total of $15$ instances. All the algorithmic parameter settings are the same as in Section \ref{subsec:expt_setting}.
}

\blue{
Figure \ref{fig:ROCPCP_time_plot} shows the time performance profiles for the C\&CG and i-C\&CG methods. The curves represent the percentage of instances solved to a relative gap of $\varepsilon=2\%$ (dotted line) or $\varepsilon=5\%$ (solid line) within a given time $t\in[1000,7200]$. Similar to our observations in Section \ref{subsec:expt_results} and \ref{appdx:additional_results}, the time profile obtained from the i-C\&CG method dominates the one obtained from the C\&CG method.  In addition, for instances that the i-C\&CG method terminates at the $2$-hour time limit, the final relative gaps are all less than $8.8\%$. However, for instances that the C\&CG method terminates at the $2$-hour time limit, the final relative gaps range from $18.9\%$ to $73.4\%$. These results further demonstrate the computational advantages using the i-C\&CG method over the C-\&CG method when the master problems are challenging. 
}
\begin{figure}[t]
    \centering
    \includegraphics[scale=0.72]{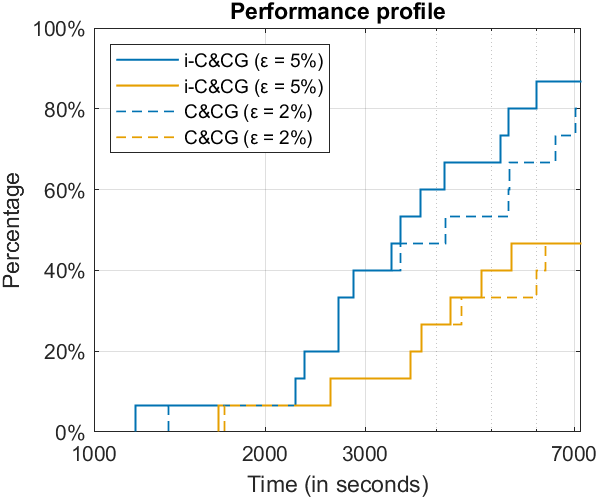}
    \caption{Time performance profile for ROCPCP}
    \label{fig:ROCPCP_time_plot}
\end{figure}

\subsection{Comparing i-C\&CG Variants}  \label{appdx:ROCPC_iCCG_variants}

\blue{
In this section, we provide examples using ROCPCP instances to illustrate the potential benefits of the i-C\&CG variants mentioned in Section \ref{subsec:inexact_C&CG_variant}. We use the same parameter settings described in Section~\ref{subsec:expt_setting}.}

\blue{First, we investigate the benefits of using the exploitation frequency. Figure \ref{fig:ROCPCP_exploit_freq_variant} shows the (actual) relative gap improvement over time of the i-C\&CG method without (the dotted line) the exploitation frequency and with $f^\text{exploit}=10$ (the solid line). The dots on each line indicate the occurrence of an exploitation step. It is clear from this figure that exploitation occurs more frequently when we impose $f^\text{exploit}=10$. In addition, since each exploitation step is followed by a large number of exploration steps (with an increased number of new scenarios in the scenario set), the lower bound significantly improves after the exploitation step. Consequently, the relative gap shrinks faster. In other words, using the i-C\&CG method with exploitation frequency could lead to a faster relative gap improvement rate than the one without exploitation frequency. This example demonstrates the potential benefits of introducing exploitation frequency in the i-C\&CG method}.

%
\begin{figure}[t]
    \centering
    \includegraphics[scale=0.72]{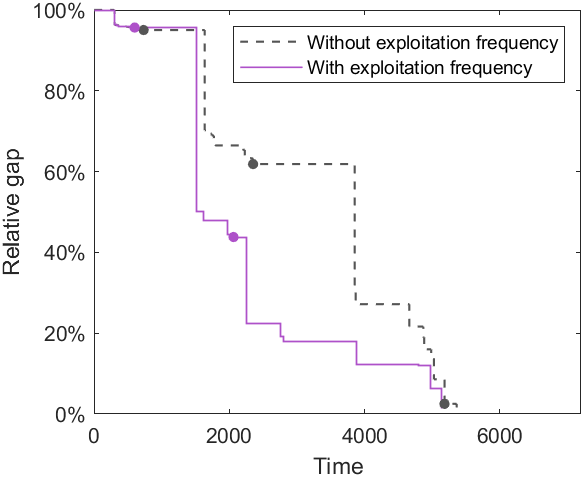}
    \caption{Relative gap improvement with and without exploitation frequency (the dots on the curves indicate the occurrence of exploitation)}
    \label{fig:ROCPCP_exploit_freq_variant}
\end{figure}

\blue{Next, we provide an example showing the benefits of imposing a time limit on the solution time of the master problem.  Figure \ref{fig:ROCPCP_time_limit_variant} shows the (actual) relative gap improvement over time of the i-C\&CG method under the same parameter settings in Section \ref{subsec:expt_setting} with (the solid line) and without (the dotted line) time limit $\tau=300$ seconds and time limit increment $\beta=600$ seconds in the exploitation step. It is clear that imposing a time limit leads to significantly faster convergence.  Moreover, we observe that, without the time limit on master problems, the i-C\&CG method attempts to solve the challenging master problem at the first iteration. On the other hand, the i-C\&CG method with a time limit could circumvent such a problem and improve the relative gap at a faster rate. This example demonstrates the potential benefits of introducing a time limit for solving the master problem in the i-C\&CG method.}


%
\begin{figure}[t]
    \centering
    \includegraphics[scale=0.72]{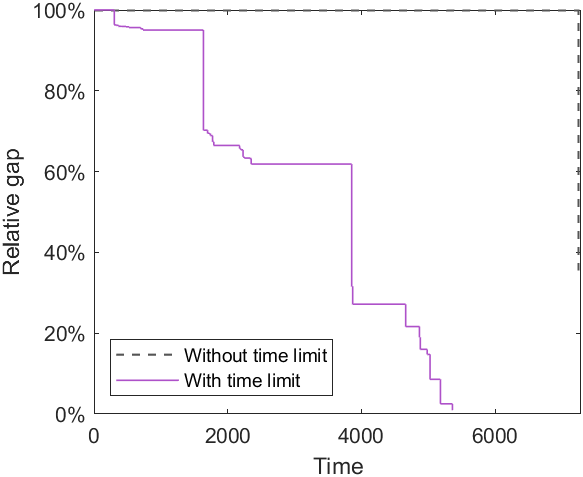}
    \caption{Relative gap improvement with and without time limit}
    \label{fig:ROCPCP_time_limit_variant}
\end{figure}


\newpage
\bibliographystyleec{elsarticle-harv}
\bibliographyec{references}


\end{document}